\documentclass[11pt,reqno]{amsart}

\usepackage{amssymb,amsmath,amsthm,amsxtra,amsfonts}

\usepackage{setspace}

\newtheorem{theorem}{Theorem}[section]

\newtheorem{lemma}[theorem]{Lemma}
\newtheorem{proposition}[theorem]{Proposition}
\numberwithin{equation}{section}
\newtheorem{definition}[theorem]{Definition}

\def\R{\mathbb R}
\def\D{\mathbb D}

\DeclareMathOperator{\imm}{Im}
\DeclareMathOperator{\diag}{diag}
\theoremstyle{remark}
\newtheorem*{remarks}{Remarks}
\newtheorem*{remark}{Remark}

\newcommand{\bbC}{{\mathbb{C}}}
\newcommand{\bbD}{{\mathbb{D}}}
\newcommand{\bbE}{{\mathbb{E}}}

\newcommand{\bbH}{{\mathbb{H}}}

\newcommand{\bbN}{{\mathbb{N}}}

\newcommand{\bbR}{{\mathbb{R}}}

\newcommand{\calJ}{{\mathcal J}}

\newcommand{\bdone}{{\boldsymbol{1}}}
\newcommand{\bdnot}{{\boldsymbol{0}}}

\newcommand{\lla}{\left\langle\!\left\langle}
\newcommand{\rra}{\right\rangle\!\right\rangle}
\newcommand{\dt}{\frac{d\theta}{2\pi}}

\newcommand{\spann}{\text{\rm{span}}}
\newcommand{\rank}{\text{\rm{rank}}}
\newcommand{\ran}{\text{\rm{Ran}\,}}

\newcommand{\beq}{\begin{equation}}
\newcommand{\eeq}{\end{equation}}
\newcommand{\ba}{\begin{align*}}
\newcommand{\ea}{\end{align*}}
\newcommand{\veps}{\varepsilon}

\newcommand{\norm}[1]{\lVert#1\rVert}

\DeclareMathOperator{\imag}{Im}

\DeclareMathOperator*{\esssup}{ess\,supp}

\DeclareMathOperator*{\res}{Res}

\usepackage[top=2in, bottom=1.5in, left=1in, right=1in]{geometry}
\begin{document}

\title[Jost asymptotics for matrix orthogonal polynomials]{Jost asymptotics for matrix orthogonal polynomials on the real line}


\author{Rostyslav Kozhan}
\date{\today}

\begin{abstract}
We obtain matrix-valued Jost asymptotics for block Jacobi matrices under an $L^1$-type condition on Jacobi coefficients, and give a necessary and sufficient condition for an analytic matrix-valued function to be the Jost function of a block Jacobi matrix with exponentially converging parameters. This establishes the matrix-valued analogue of Damanik--Simon \cite{DS2}.

The above results allow us to fully characterize the Weyl--Titchmarsh $m$-functions of Jacobi matrices with exponentially converging parameters.
\end{abstract}





\maketitle


\section{Motivation}

The main aim of this paper is to generalize some known properties from the theory of orthogonal polynomials on the real line to the matrix-valued case. The basic construction of the matrix-valued orthogonal polynomials follows the identical lines as in the scalar case. It is presented in the next section (see \cite{DPS} for a more extensive review). This  will lead us to considering the following question.
We will be studying the $l\times l$ matrix-valued solutions $(f_n(E))_{n=0}^\infty$ of
\begin{equation*}
f_{n+1}(E)A_{n}^*+f_n(E)(B_{n}-\mathbf{1}E )+f_{n-1}(E)A_{n-1}=\mathbf{0}, \quad n=1,2,\ldots
\end{equation*}
where $A_n$ are invertible $l\times l$ matrices, $B_n$ are Hermitian $l\times l$ matrices, $\bdone$ is the $l\times l$ identity matrix, and $E$ a complex number.

One of the possible solutions to this recurrence 
is the sequence of the (right) orthonormal polynomials $f_n(E)=\mathfrak{p}_{n-1}^R(E,\calJ)$ associated with the block Jacobi matrix
\begin{equation}\label{in2.5}
\mathcal{J}=\left(
\begin{array}{cccc}
B_1&A_1&\mathbf{0}& \\
A_1^* &B_2&A_2&\ddots\\
\mathbf{0}&A_2^* &B_3&\ddots\\
 &\ddots&\ddots&\ddots\end{array}\right).
\end{equation}
(see \eqref{p3_recur} below).

Another natural choice however is the unique (up to a multiplicative constant) decaying Weyl solution, which exists for all $E$ with $\imag E\ne 0$. If the matrix $\calJ$ is reasonably close to the ``free'' block Jacobi matrix $\calJ_0$ (which is, block Jacobi matrix with $A_n\equiv \bdone$, $B_n\equiv\bdnot$), then its (normalized) Weyl solution $(u_n)_{n=0}^\infty$ converges to the Weyl solution of $\calJ_0$. In this case we call $(u_n)_{n=0}^\infty$ the Jost solution (see Definition \ref{def2} below), and we say that Jost asymptotics holds. By the Jost function we will simply call the first element $u_0$ (see Definition \ref{def3} below).

Jost solution and Jost function are natural objects of study for various reasons. One of the most immediate ones is that  Jost asymptotics is essentially equivalent to the existence of the limit $z^n \mathfrak{p}^R_{n}(z+z^{-1})$ (the so-called Szeg\H{o} asymptotics). The existence of this limit has been a popular topic for many years, starting from Szeg\H{o}'s 1920 paper (see \cite{Sz}). The relevant most recent results in the scalar and matrix-valued settings can be found in \cite{PY, KS, DS1} and \cite{AN,K_sz} respectively. Another reason for interest is that properties of the Jost solution are closely related to the properties of the Weyl--Titchmarsh function
\begin{equation*}
\mathfrak{m}(z)=\int \frac{d\mu(x)}{x-z},
\end{equation*}
where $\mu$ is the ($l\times l$ matrix-valued) spectral measure of $\calJ$. This is a meromorphic Herglotz function
on $\bbC\setminus \esssup\mu$. Recall (see more details in Section 2.3) that a Herglotz function is a function satisfying $\imag m(z)>0$ if
$\imag z>0$. Conversely, any Herglotz function has the associated measure $\mu$, and it could be of interest to study the correspondence between properties of $m$ and of $\calJ$.

Jost asymptotics has been a very well studied topic for the scalar case (see \cite{DS1, DS2, GC} and references therein), but the matrix-valued analogue still lacks the complete theory.


The results of this paper can be divided into three parts.

Part I of the results (Section 3.1) deals with the \textit{direct} problem: we prove that Jost asymptotics holds under an $L^1$-type condition \eqref{el1} on the Jacobi parameters $A_n, B_n$, and establish numerous properties of the Jost function and Jost solution.

Part II of the results (Section 3.2) deals with the \textit{inverse} problem: we characterize in an if-and-only-if fashion all possible Jost functions of exponentially small perturbations of $\calJ_0$.

Finally, the results from  Part II allow us to characterize in an if-and-only-if fashion all possible Weyl--Titchmarsh functions of 
exponentially small perturbations of $\calJ_0$. 
Another, and perhaps more interesting, way of looking at it, is that we can link properties of a meromorphic Herglotz function and the asymptotic behavior of the Jacobi coefficients of the associated measure. This constitutes Part III (Section 3.3).

Some of the results in Part I already appeared in Geronimo's paper \cite{Geronimo_matrix}, and this will be mentioned and elaborated later as we state the results.

Part I and Part II follow closely the scalar analogues of Damanik--Simon \cite{DS2} (see also \cite[Chapter 13]{OPUC2}). Apart from numerous technical complications, the ideas of the proofs are borrowed from the mentioned paper.

Finally, the results of Part III appear to be new even in the scalar case.

The topic of meromorphic continuations and orthogonal polynomials on the real line is further studied in the forthcoming paper \cite{K_merom}, which deals with perturbations of periodic Jacobi matrices. The main ingredients of the proofs are the current Part III results and the ``Magic Formula'' of Damanik--Killip--Simon \cite{DKS}.

The organization of the paper is as follows. We cover some basics of matrix-valued orthogonal polynomials, Herglotz functions and matrix-valued functions in Section 2. Some auxiliary results are also collected there. In the three subsections of section 3 we state the main results corresponding to the Parts I, II, III. Then in Section 4, 5, and 6, respectively, we prove them. Note that Part I has many results scattered throughout Section 4, and it would be too space-consuming to list them all in the subsection 3.1.

\medskip

\textbf{Acknowledgements}: the results of this paper appeared as a part of the author's Ph.D thesis at the California Institute of Technology. He would like to express his gratitude to the Caltech's Math Department for all the help and hospitality, and especially to his advisor Prof Barry Simon.
The author would also like to thank Maxim Zinchenko for his observation that the use of the troublesome finite-stripping lemma (\cite[Thm 3.1]{DS2}) can be avoided when dealing with bound states. Later this lemma was eventually generalized to the matrix-valued case in the recent Simon's \cite{Simon_removal}.

This is the updated version that fixes an error in Lemma~\ref{p3_lm00} and some other minor typos and inaccuracies. The author is grateful to J. Geronimo who pointed out that there was an error in Lemma~\ref{p3_lm00}.

%
%
%
%
%


\begin{section}{Preliminaries}
\begin{subsection}{Orthogonal Polynomials on the Real Line}
We will introduce some basics of the theory of matrix-valued orthogonal polynomials on the real line here. The scalar theory is of course a special case $l=1$. We will mention some of the differences between the scalar and matrix-valued cases as we proceed.

The proofs of most of the results listed here, along with more details, can be found in the paper by Damanik--Pushnitski--Simon \cite{DPS}.

Let $\mu$ be an $l\times l$ matrix-valued Hermitian positive semi-definite finite measure on $\bbR$ of compact support, normalized
by $\mu(\bbR) = \bdone$, where $\bdone$ is the $l\times l$ identity matrix. For any $l\times l$ dimensional matrix functions $f,g$, define
\begin{align*}
\lla f,g\rra_{L^2(\mu)}&=\int f(x)^* d\mu(x) g(x);\label{p2_eq1.1}
\end{align*}
where ${}^*$ is the Hermitian conjugation.

What we have defined here is the right product of $f$ and $g$, as opposed to the left product $\int f(x)d\mu(x) g(x)^*$, whose properties are completely analogous.

Measure $\mu$ is called non-trivial if $|| \lla f,f\rra_{L^2(\mu)}||>0$ for all matrix-valued polynomials $f$. From now on assume $\mu$ is non-trivial. Then the standard arguments show that there exist unique (right) monic polynomials $\mathbf{P}^R_n$ of degree $n$ satisfying
\begin{equation*}
\lla \mathbf{P}^R_n,f \rra_{L^2(\mu)}=0 \quad \mbox{ for any polynomial } f \mbox{ with } \deg f<n.
\end{equation*}

For any choice of unitary $l\times l$ matrices $\tau_n$ (we demand $\tau_0=\bdone$), the polynomials
\begin{equation}\label{in2.3}
\mathfrak{p}^R_n=\mathbf{P}^R_n \lla \mathbf{P}^R_n,\mathbf{P}^R_n \rra_{L^2(\mu)}^{-1/2}\tau_n
\end{equation}
are orthonormal:
\begin{equation*}
\lla \mathfrak{p}^R_n,\mathfrak{p}^R_m \rra_{L^2(\mu)}=\delta_{n,m} \bdone,
\end{equation*}
where $\delta_{n,m}$ is the Kronecker $\delta$.
Using orthogonality one can show that they satisfy the (Jacobi) recurrence relation
\begin{equation}\label{p3_recur}
x \mathfrak{p}^R_n(x)=\mathfrak{p}^R_{n+1}(x)A_{n+1}^* +\mathfrak{p}^R_n(x)B_{n+1}+\mathfrak{p}^R_{n-1}(x) A_n, \quad n=1,2,\ldots,
\end{equation}
where matrices $A_n=\lla \mathfrak{p}^R_{n-1},x\mathfrak{p}^R_n \rra_{L^2(\mu)}$, $B_n=\lla \mathfrak{p}^R_{n-1},x\mathfrak{p}^R_{n-1} \rra_{L^2(\mu)}$ are called the Jacobi parameters (with $\mathfrak{p}^R_{-1}=\bdnot$, $A_0=\bdone$, the relation holds for $n=0$ too).

In the exact same fashion, just using the left product instead of right, one can define the left monic orthogonal polynomials $\mathbf{P}^L_n$ and left orthonormal polynomials $\mathfrak{p}^L_n$. It is not hard to see that $\mathbf{P}^L_n(z)=\mathbf{P}^R_n(\bar{z})^*$ and $\mathfrak{p}^L_n(z)=\mathfrak{p}^R_n(\bar{z})^*$.


Whenever we write $\mathfrak{p}_n$ without the sup-index ${}^R$ or ${}^L$, we will mean the right orthonormal polynomial $\mathfrak{p}^R_n$.

Note that if $l=1$ it is natural to choose $\tau_n=1$ in \eqref{in2.3}. In particular this gives the equality of left and right orthonormal polynomials; the Jacobi parameters become real, and $A_n$'s positive numbers. This choice of $\tau_n$'s is not necessarily the best if $l>1$. See subsection 2.2 for the further discussion.

We can arrange sequences  $\{A_n\}_{n=1}^\infty$, $\{B_n\}_{n=1}^\infty$ into an infinite matrix \eqref{in2.5}
which we call a \textit{block} Jacobi matrix, if $l>1$.

If $A_n\equiv \bdone$, $B_n\equiv \bdnot$ the corresponding block Jacobi matrix is called free.

Conversely, any block Jacobi matrix \eqref{in2.5} with invertible $\{A_n\}_{n=1}^\infty$ and Hermitian $\{B_n\}_{n=1}^\infty$ gives rise to a matrix-valued Hermitian measure $\mu$ via the spectral theorem. If $l=1$ this establishes a one-to-one correspondence between all non-trivial compactly supported measures and  bounded Jacobi matrices. If $l>1$ the same holds, except now the correspondence is with the set of \textit{equivalence classes} of bounded block Jacobi matrices (see Definition \ref{equivalent}). This has the name of Favard's Theorem (see \cite{DPS} for a proof in the matrix-valued case).

Since we will be considering perturbations of the free case in Sections 1.3.2--1.3.4, the following two classical results will prove to be useful.

\begin{lemma}[Weyl's Theorem]\label{weyl} If $A_n\to\bdone$, $B_n\to\bdnot$, then $\esssup \mu=[-2,2]$.
\end{lemma}

\begin{lemma}[Denisov--Rakhmanov Theorem]\label{denisov} Assume $\mu$ is a non-trivial $l\times l$ matrix-valued measure on $\bbR$ with associated block Jacobi matrix $\calJ$ of type $3$ such that $\esssup \mu=[-2,2]$ and $\det\left(\frac{d\mu(x)}{dx}\right)>0$ a.e. on $[-2,2]$. Then $A_n\to\bdone$, $B_n\to\bdnot$.
\end{lemma}

Definition \ref{types} below explains what it means for a Jacobi matrix to be of type $3$.

Lemma \ref{weyl} is trivial, while Lemma \ref{denisov}, in the form given here, is proven in \cite{DKS} (see also \cite{yakhlef-marc}, as well as \cite{denisov, rakhmanov}).

Define the (Weyl-Titchmarsh) \textbf{${m}$-function} of the measure $\mu$ to be the meromorphic in $\bbC\setminus \esssup\mu$ matrix-valued function
\begin{equation}\label{p3_m-func}
\mathfrak{m}(z)=\int \frac{d\mu(x)}{x-z}.
\end{equation}

Define $\mathcal{J}^{(1)}$ to be the ``once-stripped'' Jacobi matrix with the Jacobi parameters $\{A_n,B_n\}_{n=2}^\infty$, i.e., the Jacobi matrix of the form  \eqref{in2.5} with the first row and column removed. Then the following holds (the matrix-valued version is due to Aptekarev--Nikishin \cite{AN}):
\begin{equation}\label{m_recur}
A_1 \mathfrak{m}(z;\calJ^{(1)}) A_1^*= B_1-\bdone z-\mathfrak{m}(z;\calJ)^{-1}.
\end{equation}

Similarly one defines the $k$ times stripped Jacobi matrix $\mathcal{J}^{(k)}$ to be the Jacobi matrix with first $k$ columns and $k$ rows removed, i.e., the Jacobi matrix with the Jacobi parameters $\{A_n,B_n\}_{n=k+1}^\infty$.

We will use the following result. This is proven in \cite{denisov} for the scalar case, and appears in \cite{DPS} for the matrix-valued case.
\begin{lemma}\label{stripping2}
Let $\sigma_{ess}(\calJ)\subseteq [-2,2]$. Then, for every $\veps>0$, there exists $N$ such that for $n\ge N$, we have that $\sigma(\calJ^{(n)})\subseteq[-2-\veps,2+\veps]$.
\end{lemma}

\end{subsection}

\begin{subsection}{Equivalence Classes of Block Jacobi Matrices}
\begin{definition}\label{equivalent}
Two block Jacobi matrices $\calJ$ and $\widetilde{\calJ}$ are called \textbf{equivalent} if they correspond to the same spectral measure $\mu$ $($but a different choice of $\tau_n$'s in \eqref{in2.3}$)$.
\end{definition}

They are equivalent if and only if their Jacobi parameters satisfy
\begin{equation}\label{p1_eq2.3}
\widetilde{A}_n=\sigma_n^*  A_n \sigma_{n+1}, \quad
\widetilde{B}_n=\sigma_n^*  B_n \sigma_n
\end{equation}
for unitary $\sigma_n$'s with $\sigma_1=\bdone$ (the connection with $\tau_j$'s is $\sigma_n=\tau_{n-1}^*  \widetilde{\tau}_{n-1}$). It is easy to see that
\begin{equation}\label{p1_eq2.4}
\widetilde{\mathfrak{p}}^R_n(x)=\mathfrak{p}^R_n(x)\sigma_{n+1},
\end{equation}
where $\widetilde{\mathfrak{p}}_n$ are the orthonormal polynomials for $\widetilde{\calJ}$ associated with the Jacobi parameters $\{\widetilde{A}_n\}_{n=1}^\infty$, $\{\widetilde{B}_n\}_{n=1}^\infty$.

\begin{definition}\label{types}
A block Jacobi matrix is of \textbf{type $1$} if $A_n>0$ for all $n$,
of \textbf{type $2$} if $A_1A_2\ldots A_n>0$ for all $n$, and of \textbf{type $3$} if every $A_n$ is lower triangular with strictly positive elements on the diagonal.
\end{definition}

Each equivalence class of block Jacobi matrices contains exactly one matrix of type $1$, $2$, and $3$ (follows from the uniqueness of the polar and QR decompositions, see \cite{DPS} for the proof).

\begin{definition}
We say that $\calJ$ is in the \textbf{Nevai class} if
\begin{equation*}
B_n\to\bdnot, \quad A_nA_n^* \to \bdone.
\end{equation*}
\end{definition}

Note that this definition is invariant within the equivalence class of Jacobi matrices. Then

\begin{lemma}\label{p1_thm1}
Assume $\calJ$ belongs to the Nevai class. If $\calJ$ is of type $1$, $2$, or $3$, then $A_n\to\bdone$ as $n\to\infty$.
\end{lemma}

This result was proven in \cite{DPS} for the type $1$ and $3$ cases, and in \cite{K_equiv} for type $2$.


Note that since we are interested in the asymptotics of the orthonormal polynomials as $n\to\infty$, and because of the relation~\eqref{p1_eq2.4},  it is desirable to know when $\lim_{n\to\infty} \sigma_n$ exists. This explains the need of the following definition.

\begin{definition}\label{asymptotic}
Two equivalent matrices $\calJ$ and $\widetilde{\calJ}$ with~\eqref{p1_eq2.3} are called {\textbf{asymptotic}} to each other if the limit $\lim_{n\to\infty} \sigma_n$ exists.
\end{definition}

Clearly this is an equivalence relation on the class of equivalent Jacobi matrices. Note that establishing asymptotics for orthonormal polynomials automatically establishes the corresponding asymptotics for the polynomials corresponding to any Jacobi matrix asymptotic to the original one. 

The following was proved in \cite{K_equiv}
\begin{lemma}\label{p1_thm2}
Assume
\begin{equation*}\label{p1_eq2.6}
\sum_{n=1}^\infty \left[ \norm{\bdone-A_n A_n^*}+ \norm{B_n}\right]<\infty.
\end{equation*}
Then the corresponding Jacobi matrices of type $1$, $2$, and $3$ are pairwise asymptotic.
\end{lemma}

It was also shown in \cite{K_equiv} that any equivalent Jacobi matrix, for which eventually each $A_n$ has real eigenvalues, is also asymptotic to type $1$, $2$, $3$.
%
%
\end{subsection}


\begin{subsection}{Herglotz Functions}

\begin{definition} An analytic in $\bbC_+\equiv\{z:\imag z>0\} $ $l\times l$ matrix-valued function $m$ is called \textbf{Herglotz} if $\imag m(z)\ge \bdnot$ for all $z\in  \bbC_+$.
\end{definition}
Here $\imag T\equiv \frac{1}{2i}(T-T^*) $.

We can also define $m$ on the lower half plane $\bbC_-$ by reflection $m(z)=m(\bar{z})^*$, so that $\imag m(z) \le \bdnot$ for all $z$ with $\imag z<0$. In particular the ${m}$-function $\mathfrak{m}$ defined in \eqref{p3_m-func} is Herglotz.

We will assume from now on that $\det\imag m(z)$ is not identically zero, in which case  the inequality in $\imag m(z)\gtrless \bdnot$ is everywhere strict (see \cite[Lemma 5.3]{Ges}).

The following result is well-known (see, e.g., \cite[Thm 5.4]{Ges}).

\begin{lemma}\label{p4_ges}
Let $m$ be an $l\times l$ matrix-valued Herglotz function. Then there exist an $l \times l$ matrix-valued measure $\mu$ on $\bbR$ satisfying $\int_\bbR \frac{1}{1+x^2}d\mu(x)<\infty$, and constant matrices $C=C^*, D\ge\bdnot$ such that
\begin{equation*}\label{p4_herg}
m(z)=C+Dz+\int_\bbR \left(\frac{1}{x-z}-\frac{x}{1+x^2}\right) d\mu(x), \quad z\in \bbC_+.
\end{equation*}
The absolutely continuous part of $\mu$ can be recovered from this representation by
\begin{equation}\label{p2_herg1}
f(x)\equiv\frac{d\mu}{dx}=\pi^{-1}\lim_{\veps\downarrow0} \imag m(x+i \veps),
\end{equation}
and the pure point part by
\begin{equation}\label{p2_herg2}
\mu(\{\lambda\})=\lim_{\veps \downarrow 0} \veps\, \imag m(\lambda+i\veps) = \lim_{\veps\downarrow 0} \veps\, m(\lambda+i\veps).
\end{equation}
\end{lemma}

\begin{definition}\label{disctrete}
A \textbf{discrete $m$-function} is a Herglotz function, $m(z)$, which has an analytic continuation from $\bbC_+$ to $\bbC\setminus I$ for some bounded interval $I\subset \bbR$, and satisfies
\begin{align*}
&z\in\bbR\setminus I \Rightarrow \imag m(z)=\bdnot, \\
&m(z)=z^{-1}\bdone +O(z^{-2}) \mbox{ at }\infty.
\end{align*}
\end{definition}

The following is immediate from Lemma \ref{p4_ges}.
\begin{lemma}
A function $m(z)$ on $\bbC_+$ is a discrete $m$-function if and only if
\begin{equation*}
m(z)=\int_\bbR \frac{d\mu(x)}{x-z}
\end{equation*}
for some probability measure $\mu$ on $\bbR$ with bounded support.
\end{lemma}
\end{subsection}

\subsection{Matrix-Valued Functions}

Throughout the paper, all meromorphic and analytic matrix functions are assumed to have  not identically vanishing  determinant.

The order of a pole of an $l\times l$ matrix-valued meromorphic function $f$ is defined to be the minimal $k>0$ such that $\lim_{z\to z_0} (z-z_0)^k f(z)$ is a finite nonzero matrix. A simple pole is a pole of order $1$.

By a zero of a matrix-valued meromorphic function $f$ we call a point at which $f^{-1}$ has a pole. A simple zero of $f$ is a point where $f^{-1}$ has a simple pole.


We will make use of the so-called (local) Smith--McMillan form (see, e.g.,~\cite[Thm 3.1.1]{Ball}).

\begin{lemma}\label{p4_lmSM}
Let $f(z)$ be an $l\times l$ matrix-valued function meromorphic at $z_0$ with determinant not identically zero. Then $f(z)$ admits the representation
\begin{equation*}\label{p4_ee1.31}
f(z)=E(z) \diag\left((z-z_0)^{\kappa_1},\ldots,(z-z_0)^{\kappa_l}\right) F(z),
\end{equation*}
where $E(z)$ and $F(z)$ are $l\times l$ matrix-valued functions which are analytic and invertible in a neighborhood of $z_0$, and $\kappa_1\ge \kappa_2 \ge \ldots \ge\kappa_l$ are integers (positive, negative, or zero).
\end{lemma}

This immediately gives us the following corollary.

\begin{lemma}\label{p4_lm0}
Let $u$ be an analytic function at $z_0$ such that $z_0$ is a zero of $\det u$ of order $k>0$. Then $\dim \ker u(z_0)=k$ if and only if $z_0$ is a pole of $u(z)^{-1}$ of order $1$. 

If this is the case, then
\begin{align*}
\ker\res_{z=z_0} u(z)^{-1} &=\ran u(z_0), \\
\ran \res_{z=z_0} u(z)^{-1} &=\ker u(z_0).
\end{align*}
\begin{proof}
Both of the conditions in the if-and-only-if statement are equivalent to saying that $\kappa_1=\ldots=\kappa_k=1$, $\kappa_{k+1}=\ldots=\kappa_l=0$ in the Smith--McMillan form of $u(z)$ at $z_0$. Then note that both $\ker\res_{z=z_0} u(z)^{-1}$ and $\ran u(z_0)$ are equal to $E(z_0)\spann\left\{\delta_{k+1},\cdots,\delta_l\right\}$. Similarly, one sees that both $\ran \res_{z=z_0} u(z)^{-1} $ and $\ker u(z_0)$ are equal to $F(z_0)^{-1} \spann\left\{\delta_{1},\cdots,\delta_k\right\}$.
\end{proof}
\end{lemma}

We will also need some facts about matrix-valued outer functions and matrix-valued Blaschke--Potapov products.

\begin{lemma}[Wiener--Masani \cite{WM}]\label{p2_lm2}
Suppose $w(\theta)$ is a non-negative matrix-valued function on the unit circle satisfying
\begin{equation*}
\int_{-\pi}^\pi \log \det w(\theta)\dt>-\infty.
\end{equation*}
Then there exists a unique matrix-valued $H^2(\bbD)$ function $G(z)$ satisfying
\begin{gather*}
G(e^{i\theta})^* G(e^{i\theta})=w(\theta),\\
G(0)^*=G(0)>0,
\end{gather*}
\begin{equation}
\log|\det G(0)|=\int_{-\pi}^{\pi} \log|\det G(e^{i\theta})|\dt. \label{p2_eq2.21}
\end{equation}
\end{lemma}

This is a well-known result of Wiener--Masani \cite{WM}. The proof of the uniqueness part can be found, e.g., in~ \cite{DGK}.

Equality~(\ref{p2_eq2.21}) implies (see \cite[\S 17.17]{Rud}) that $\det G(z)$ is a scalar outer function, which implies (by definition) that $G(z)$ is a matrix-valued outer function. 

%
%

The Blaschke--Potapov elementary factor is a generalization of scalar Blaschke factors:
$$
B_{z_j,s,U}(z)=U^*\left(
\begin{array}{cccccc}
\frac{|z_j|}{z_j}\frac{z_j-z}{1-z_j z}&0&0&0&\cdots&0\\
0&\ddots&0&0&\cdots&0\\
0&0&\frac{|z_j|}{z_j}\frac{z_j-z}{1-z_j z}&0&\cdots&0\\
0&0&0&1&\cdots&0\\
\vdots&\vdots&\vdots& &\ddots& \\
0&0&0&0&\cdots&1
\end{array}\right)U, \quad z\in\bbD,
$$
where $z_j\in\bbD$, $s$ is the number of the scalar Blaschke factors on the diagonal ($0\le s\le l$), and $U$ is a unitary constant matrix. Clearly $B_{z_j,s,U}$ is an analytic  in $\bbD$ function with unitary values on the unit circle.

%

We will need to following (slightly modified) lemma from \cite{K_sz}:

\begin{lemma}\label{product}
Let $\{z_k\}_{k=1}^\infty$ with $\sum_{k=1}^\infty (1-|z_k|)<\infty$ be given, with all $z_k$ pairwise different. For any sequence of subspaces $V_k\subseteq \bbC^l$, there exists a unique product $B(z)=\stackrel{\curvearrowright}{\prod_{j=1}^\infty} B_{z_j,s_j,U_j}(z)$ for some choice of numbers $s_k$, $0\le s_k\le l$, and unitary matrices $U_k$, that  satisfies
\begin{equation*}\label{p2_eq2.28}
\ker B(z_k)=V_k \quad \mbox{for all } k.
\end{equation*}
\end{lemma}



\subsection{Miscellaneous Lemmas}
Recall that an infinite product $\prod_{j=1}^\infty a_j$ with $a_j\ne 0$ is called absolutely convergent if $\sum_{j=1}^\infty |1-a_j|<\infty$. We will be needing the following easy statements.

\begin{lemma}\label{p3_lm3}
\begin{itemize}
\item[(i)] If $\prod_{j=1}^\infty a_j$ with $a_j\ne 0$ is absolutely convergent then
\begin{equation*}
\sup_{\Lambda\subset\bbN} \left| \prod_{j\in\Lambda} a_j \right|<\infty.
\end{equation*}
\item[(ii)] Let $a_n\to 0$ and $\sum_{j=1}^\infty |b_j|<\infty$. Then
\begin{equation*}
\sum_{j=0}^n a_{n-j} b_j \to 0.
\end{equation*}
\end{itemize}
\begin{proof}
(i) If $\prod_{j=1}^\infty a_j$ is absolutely convergent, then so is $\prod_{j=1}^\infty |a_j|$, so without loss of generality we can assume $a_j>0$. Then
\begin{equation*}
\prod_{j\in\Lambda} a_j =e^{\sum_{j\in\Lambda} \log a_j} \le e^{\sum_{j\in\Lambda} |a_j-1|} \le e^{\sum_{j=1}^\infty |a_j-1|}<\infty.
\end{equation*}

\medskip

(ii) For any $\varepsilon>0$ find $N$ such that $|a_j|<\varepsilon$ for all $j \ge N$. Then for $n>N$:
$$
\left| \sum_{j=0}^n a_{n-j} b_j\right|\le \left| \sum_{j=0}^N a_{n-j} b_j \right| + \varepsilon \sum_{j=N+1}^n \left| b_j \right| \le \left| \sum_{j=0}^N a_{n-j} b_j \right| + \varepsilon \sum_{j=1}^\infty\left| b_j \right|,
$$
which implies $\limsup_{n\to\infty} \left| \sum_{j=0}^n a_{n-j} b_j\right| \le \varepsilon \sum_{j=1}^\infty\left| b_j \right|$, and proves (ii).
\end{proof}
\end{lemma}

\begin{remark}
Note that part (ii) works also for the matrix-valued $a$'s and $b$'s.
\end{remark}

\begin{lemma}\label{p3_lm00}
		Let $A$ and $B$ be two $l\times l$ matrices. 

There exists a 
nonnegative definite $l\times l$ matrix $W$ satisfying
\begin{align}
	& \label{eq:matrixEq1} WA = B \\
	& \label{eq:matrixEq2} \ran W = \ran B
\end{align}
if and only if the following three conditions hold:
\begin{align}
	& \label{eq:cond1} \ker A \subseteq \ker B ,
	\\
	& \label{eq:cond2}A^* B = B^* A \ge 0, \\
	& \label{eq:cond3} \ran B \cap \ker(A^*) = \{0\}.
\end{align}
Moreover, the  solution is then unique and given by
\begin{equation}\label{eq:W}
	W  
	= B(B^* A)^+ B^*,
\end{equation}
where $X^+$ stands for the Moore--Penrose inverse of $X$. 
\end{lemma}
\begin{remarks}
	1. Recall that the Moore--Penrose inverse of $X$ is the unique matrix $X^+$ of the same size as $X$ such that
	\begin{align}
		& \label{eq:MP1} (XX^+)^* = XX^+, \\
		& \label{eq:MP2} (X^+ X)^* = X^+ X, \\
		& \label{eq:MP3} XX^+ X = X, \\
		& \label{eq:MP4}  X^+ X X^+ = X^+.
	\end{align}
	It is uniquely defined for any matrix $X$, and it coincides with $X^{-1}$ if $X$ happens to be invertible. 
	
	2. Necessary and sufficient conditions~\eqref{eq:cond1},~\eqref{eq:cond2},~\eqref{eq:cond3} for the matrix linear equation~\eqref{eq:matrixEq1} to have nonnegative definite solutions (without the extra requirement~\eqref{eq:matrixEq2}) were established by Khatri--Mitra~\cite{KhaMit}. 
	 The only new result in this lemma compared to ~\cite{KhaMit} is that condition~\eqref{eq:matrixEq2} ensures uniqueness. We provide here the full proof for completeness purposes, with some ideas borrowed from~\cite{DajKol08} which contains a nice review and further references.
	
	3. Conditions ~\eqref{eq:cond1} and \eqref{eq:cond3} have a different but equivalent form compared to the ones in~\cite{KhaMit}. See Lemma~\ref{lem:temp}.
\end{remarks}

\begin{lemma}\label{lem:temp}
	Let $A$ and $B$ be two $l\times l$ matrices.
	\begin{itemize}
		\item[(i)] $\ker A \subseteq \ker B$ if and only if $BA^+  A= B$.
		\item[(ii)] $\ran B \cap \ker(A^*) = \{0\}$ if and only if $\rank(A^* B) = \rank B$.
	\end{itemize}
\end{lemma}
\begin{proof}
	(i) That $BA^+ A = B$ implies $\ker B = \ker[BA^+ A] \supseteq \ker A$ is trivial. Conversely, suppose $\ker A \subseteq \ker B$. It is well known that $A^+ A$ is the orthogonal projection onto $\ran A^* = (\ker A)^\perp$. 
	Therefore $BA^+ A = B$ holds on $(\ker A)^\perp$. For $v\in\ker A$ we get also $v\in \ker B$, so that both $BA^+ A v=0$ and $Bv=0$. So indeed $BA^+ A = B$.
	
	(ii) Both conditions are equivalent to $\ker(A^* B) = \ker B$.
\end{proof}

\begin{proof}
	[Proof of Lemma~2.19]
	Suppose \eqref{eq:matrixEq1}--\eqref{eq:matrixEq2} has a unique nonnegative definite solution $W$. Then $WAA^+A = BA^+A$. This combined with~\eqref{eq:matrixEq1} and~\eqref{eq:MP3} gives $B = BA^+A$ which is~\eqref{eq:cond1} by Lemma~\ref{lem:temp}(i). Further, $W\ge 0$  implies $A^* W A = A^* B$ is also nonnegative definite. Finally, \eqref{eq:matrixEq2} implies that Then third also follows: $\rank(A^*B)=\rank(A^*WA) = \rank(A^*W)=\rank B^* = \rank B$. Here we used that $\rank(A^*WA) = \rank(A^*W)$ which is equivalent to $\ker(A^*WA) = \ker(WA)$ which follows from 
	$$
	v\in\ker(A^*WA) \Rightarrow A^*WA v =0 \Rightarrow
	||W^{1/2} A v ||=0
	\Rightarrow W^{1/2}W^{1/2} A v =0\Rightarrow v\in\ker(WA).
	$$
	
	Conversely, suppose \eqref{eq:cond1}, \eqref{eq:cond2}, \eqref{eq:cond3} hold. Define $W$ as in~\eqref{eq:W}. It is nonnegative definite by~\eqref{eq:cond2}. Let us show it solves \eqref{eq:matrixEq1}--\eqref{eq:matrixEq2}.

	By~\eqref{eq:cond3} and Lemma~\ref{lem:temp}(ii) we get $\ker(A^*B) \subseteq \ker B$ which by Lemma~\ref{lem:temp}(i) is  equivalent to $B(A^*B)^+A^*B = B$. By~\eqref{eq:cond2} this can be rewritten as $B(B^* A)^+ B^* A$ which is ~\eqref{eq:matrixEq1}. Clearly $\ran W = \ran(B(B^* A)^+ B^* )\subseteq \ran B$. But $WA= B$ implies $\ran W \supseteq \ran B$, which proves~\eqref{eq:matrixEq2}.
	
	Finally, we need to show uniqueness of $W$. $W$ maps $\ran A$ onto $\ran B$ and there it is uniquely defined by $WA = B$. $W$ is also uniquely defined on $(\ran B)^\perp = (\ran W)^\perp = (\ran W^*)^\perp = \ker W$ to be zero. Therefore $W$ is uniquely determined on the space $\ran A + (\ran B)^\perp$ whose dimension is 
	$$
	\dim \ran A + \dim (\ran B)^\perp - \dim \ran A\cap(\ran B)^\perp.
	$$
	Denote $\dim\ker A = n_A$ and $\dim\ker B = n_B$. By \eqref{eq:cond1}, $n_A \le n_B$. Bby the rank-nullity theorem $\dim\ran A = l-n_A$, $\dim (\ran B)^\perp 
	=n_B$. 
	
	Now, $\ran B = \ran |WA = \ran W$ means that $\rank B = \rank(WA) = \rank A - \dim \ran A\cap \ker W$, so that  $\dim \ran A\cap(\ran B)^\perp =  \dim \ran A\cap(\ker W) = \rank A - \rank B = n_B-n_A$. This leads to $\ran A + (\ran B)^\perp$ having dimension $l-n_A + n_B-(n_B-n_A) = l$. So $W$ is uniquely determined on  the whole $\bbC^l$.
\end{proof}



\end{section}


\section{Main Results}

\subsection{Part I. Direct Problem}


As was mentioned in the introduction, we are interested in the $l\times l$ matrix-valued solutions $(f_n(E))_{n=0}^\infty$ of
\begin{equation}\label{p3_eq1}
f_{n+1}(E)A_{n}^*+f_n(E)(B_{n}- \mathbf{1}E)+f_{n-1}(E)A_{n-1}=\mathbf{0}, \quad n=1,2,\ldots
\end{equation}

By~\eqref{p3_recur}, one solution of this is $f_n(E)=\mathfrak{p}_{n-1}^R(E,\calJ)$.

\begin{definition}\label{def1} For any two sequences $(v_n)_{n=0}^\infty$, $(w_n)_{n=0}^\infty$  their \textbf{Wronskian} is
$$
W_n(v,w;\calJ)=v_nA_nw_{n+1}-v_{n+1}A_n^* w_n.
$$
\end{definition}
If $v_n(E)$ and $w_n(E)$ both solve~(\ref{p3_eq1}), then $W_n(v_n(E),w_n(\bar{E})^*)$ is independent of $n$ (see~\cite{DPS}).

In this subsection we will be considering only $\calJ$ with $\esssup\mu=[-2,2]$, so it will be convenient to move from $\bbC\setminus[-2,2]$ to $\bbD$ via $z+z^{-1}=E$.

\begin{definition}\label{def2} The \textbf{Jost solution}, $\{u_n(z;\calJ)\}_{n=0}^\infty$, is a solution of \eqref{p3_eq1} with
\begin{equation}\label{p3_eq2}
z^{-n} u_n(z;\calJ)\to \bdone
\end{equation}
as $n\to\infty$, where $z+z^{-1}=E$.
\end{definition}
In general there may or may not be a solution of \eqref{p3_eq1} satisfying~\eqref{p3_eq2}, though there always exists an $\ell^2$ (Weyl's) solution of \eqref{p3_eq1} for $z\in\bbD$.

\begin{definition}\label{def3}
If the Jost solution exists (it is then unique, of course), then the \textbf{Jost function}  is defined to be
\begin{equation*}
u(z;{\calJ})=W(u_{\cdot}(z;{\calJ}),\mathfrak{p}_{\cdot-1}^L(z+z^{-1};{\calJ}))=u_0(z;\calJ),
\end{equation*}
where $\mathfrak{p}_n^L(z)$ are left orthonormal polynomials of $\calJ$.
\end{definition}
The last equality here comes from the constancy of the Wronskian.

In Section 4 we establish that the Jost solution and Jost function exist for block Jacobi matrices asymptotic to type $1$ (which includes type $2$ and $3$) under the condition
\begin{equation} \label{el1}
\sum_{n=1}^\infty \left[ ||B_n||+||\mathbf{1}-A_nA_n^*||\right]<\infty,
\end{equation}
and establish a number of their properties. See Theorems \ref{p3_th1}, \ref{p3_th5}, \ref{p3_th7000}. Some of there results (namely, Theorem \ref{p3_th1} and parts (iv)--(vi) of Theorem  \ref{p3_th5}) were already proven in Geronimo's \cite{Geronimo_matrix}.

These results also give us the following corollaries. Firstly, it's Theorem \ref{p3_th2}, which is Szeg\H{o}'s asymptotics under the $L^1$-type condition on the Jacobi coefficients. This result was already obtained by other methods in \cite{K_sz}, where Szeg\H{o} asymptotics is established in a larger generality. Second corollary is Theorem \ref{p3_th7000}, which is the matrix-valued analogue of a theorem in the Killip--Simon's paper \cite{KS}. It says that under the $L^1$-condition, the Jost function has trivial singular inner part.


\subsection{Part II. Inverse Problem}

Here we deal with the inverse direction.

Recall that zeros of a matrix-valued function $f$ are defined to be the poles of its inverse $f^{-1}$. A zero of $f$ is simple if the corresponding pole of $f^{-1}$ is simple.

First we prove

\begin{theorem}\label{p3_th_constr}
Let $u$ be an analytic function in a disk $\bbD_R=\{z\mid |z|<R\}$ for some $R>1$, satisfying
\begin{equation}\label{eq:symmetry}
	u(1/\bar{z})^* u(z) = \left(u(1/{z})^* u(\bar z)\right)^*,
\end{equation}
whose only zeros in $\overline{\bbD}$ lie in $(\overline{\D}\cap\R)\setminus\{0\}$ each being simple. For each zero $z_j$ in $({\D}\cap\R)\setminus\{0\}$, let a nonzero matrix-valued weight $w_j\ge0$ be given so that
\begin{itemize}
\item[(i)] $\sum_j w_j+ \frac2\pi \int_{0}^\pi \sin^2\theta\,\left[u(e^{i\theta})^* u(e^{i\theta})\right]^{-1}d\theta=\mathbf{1}$
\item[(ii)] $\ran w_j=\ker u(z_j)$ for all $j$.
\end{itemize}
Then there exists a unique measure $d\mu$ for which $w_j$ are the weights and $u$ is its Jost function for some choice of Jacobi matrix from the equivalence class corresponding to $d\mu$. Any such matrix is of type asymptotic to $1$.
\end{theorem}

Note that the conditions in this theorem are also necessary in view of Theorem \ref{p3_th5}.
Now that we established the existence of the measure $\mu$,  we can further specify the properties of $u$ which correspond (in an if-only-if fashion) to the prescribed exponential decay of the Jacobi parameters.

We will need the following definition, after which we will state the last two main theorems of the section.

\begin{definition}\label{canonic}
Let $u$ satisfy the conditions of Theorem \ref{p3_th_constr}. Suppose $u$  has a zero at some $1>|z_j|>R^{-1}$, $\ran w_j=\ker u(z_j)$. The weight  $w_j$ is said to be \textbf{canonical} if
\begin{equation*}\label{p3_eq53}
\frac{z_j}{z_j^{-1}-z_j}{w}_j \, u(1/{\bar{z}_j})^*=-(z_j-z_j^{-1})\lim_{z\to z_j}(z-z_j)u(z)^{-1}.
\end{equation*}
\end{definition}

\begin{theorem}\label{p3_th10}
If a polynomial $u(z)$  satisfying~\eqref{eq:symmetry} obeys
\begin{itemize}
\item[(i)] $u(z)$ is invertible on $(\overline{\D}\setminus \R)\cup\{0\}$;
\item[(ii)] all zeros on $\overline{\D}\cap\R$ are simple;
\item[(iii)]$\sum_jw_j+\frac2\pi \int_{0}^\pi \sin^2\theta\,\left[u(e^{i\theta})^* u(e^{i\theta})\right]^{-1}d\theta=\mathbf{1}$ for some $w_j\ge0$, $\ran w_j=\ker u(z_j)$ for each zero $z_j$ of $u$ in $\D\cap\R$,
\end{itemize}
then $u$ is the Jost function for a Jacobi matrix with exponentially converging parameters. It has  $\mathbf{1}-A_nA_n^*=B_n=\mathbf{0}$ for all large $n$ if and only if all the weights are canonical.
\end{theorem}

\begin{remarks}
1. It's easy to see that if the degree of $u$ is $k$, then $\mathbf{1}-A_nA_n^*=B_n=\mathbf{0}$ holds for $n\ge \lceil \frac{k}{2} \rceil + 1$ (follows from Theorem \ref{p3_th1}).

2. Condition (iii) assumes existence of nonnegative definite matrix $w_j$ that solves a linear matrix system of the form~\eqref{eq:matrixEq1}--\eqref{eq:matrixEq2}, see Lemma~\ref{p3_lm00} for the explicit necessary and sufficient conditions on $u$, and the formula for $w_j$ in terms of $u$. 
\end{remarks}

\begin{theorem}\label{p3_th11}
Let $u(z)$ be analytic in $\bbD_R$ for some $R>1$ and obeys~\eqref{eq:symmetry}, (i), (ii), (iii) of Theorem~{\ref{p3_th10}}. Then $u$ is the Jost function for a Jacobi matrix with exponentially converging parameters. It has
\begin{equation*}\label{p3_eq4.35}
\limsup_{n\to\infty}\left(||B_n||+||\mathbf{1}-A_nA_n^*||\right)^{1/2n}\le R^{-1}
\end{equation*}
if and only if all weights for $z_j$ with $1>|z_j|>R^{-1}$ are canonical.
\end{theorem}

\begin{remark}
By ``exponentially converging parameters'' it is meant that they satisfy $$\limsup_{n\to\infty}\left(||B_n||+||\mathbf{1}-A_nA_n^*||\right)^{1/2n}\le r^{-1}$$ for some $r$  (in general $r=\min_j\{|z_j|^{-1}\}$, unless some of the weights are canonical).
\end{remark}




\begin{subsection}{Part III. Meromorphic Continuations of Matrix Herglotz Functions and Perturbations of the Free Case}
As we mentioned earlier, the results of Part III are new even for the scalar case $l=1$. Note that in this setting, (D) of Theorems \ref{p3_th17} and \ref{p3_th18} is reduced to the much simpler condition that $M$ has no simultaneous singularities at points $z_j$ and $z_j^{-1}$ (see Proposition \ref{propos}).

We will consider measures $\mu$ with essential support one interval.  By scaling and translating we can assume that $\esssup\mu=[-2,2]$. Instead of discussing  meromorphic continuations of $\mathfrak{m}$ (see \eqref{p3_m-func}) through $(-2,2)$, it will be convenient to move $\bbC\setminus [-2,2]$ to $\bbD$  via the inverse of $z \mapsto z+z^{-1}$, and discuss the meromorphic continuations of
\begin{equation}\label{EM}
M(z)=-\mathfrak{m}(z+z^{-1})
\end{equation}
from $\bbD$ through $\partial\bbD$. Note that $M$ is also Herglotz in the meaning that $\imag M(z)\gtrless \bdnot$ if $z\in\bbC_{\pm}\cap \bbD$.

Let us use the notation $M^\sharp(z)=M(\bar{z}^{-1})^*$.

Note that any Herglotz function $m$ has an associated measure $\mu$ (Lemma \ref{p4_ges}), which has an associated class of equivalent block Jacobi matrices.

We prove the following result.

\begin{theorem}\label{p3_th17}
Let $\mathfrak{m}$ be a discrete $l\times l$ matrix-valued $m$-function, and $M$ is given by \eqref{EM}. Let $R>1$. The following are equivalent:
\begin{itemize}
\item[(I)] The corresponding to $\mathfrak{m}$ Jacobi matrices $\{A_n,B_n\}_{n=1}^\infty$ satisfy
    \begin{equation*}\label{p3_eq5.3}\limsup_{n\to\infty}\left(||B_n||+||\mathbf{1}-A_nA_n^*||\right)^{1/2n}\le R^{-1}.\end{equation*}
\item[(II)] All of the following holds:
\begin{itemize}
\item[(A)] $M$ has a meromorphic continuation to $\bbD_R$.
\item[(B)] $M$ has no poles on $\partial\bbD\setminus\{\pm1\}$, and at most simple poles at $\pm1$.
\item[(C)] $(M(z)-M^\sharp(z))^{-1}$ has no poles in $R>|z|>R^{-1}$ except at $z=\pm1$ where there might be simple poles.
\item[(D)] If $M$ has a pole at $z_j\in \{z: R^{-1}<|z|<1\}$ and at $z_j^{-1}$, then 
\begin{gather}
\label{p3_eq5.4}
\ran \res_{z=z_j}M(z) \subseteq \ker (M(z_j^{-1})-M^\sharp(z_j^{-1}))^{-1} ,\\
\label{p3_eq5.5}
\ran \res_{z=z_j}M(z) \subseteq \left(\ran (M(z^{-1}_j)-M^\sharp(z^{-1}_j))^{-1} M(z^{-1}_j)\right)^\perp.
\end{gather}
\end{itemize}
\end{itemize}
\end{theorem}

Note that $R=\infty$ is allowed, in which case (I) states that the decay of the Jacobi coefficients is super-exponential, while in (II) $M$ is meromorphic in $\bbC$. We can also demand that $M$ is actually meromorphic in $\bbC\cup\{\infty\}$ (which, of course, is the same as saying that $M$ is a rational matrix function). This corresponds to strengthening the condition (I) to \eqref{p3_eq5.8}. Therefore we are able to characterize all possible $M$-functions of eventually-free Jacobi matrices.

\begin{theorem}\label{p3_th18}
Let $\mathfrak{m}$ be a discrete $l\times l$ matrix-valued $m$-function, and $M$ is given by \eqref{EM}.
The following are equivalent:
\begin{itemize}
\item[(I)] The corresponding to $\mathfrak{m}$ Jacobi matrices $\{A_n,B_n\}_{n=1}^\infty$ satisfy
    \begin{equation}\label{p3_eq5.8}||B_n||+||\mathbf{1}-A_nA_n^*||=\bdnot \quad \mbox{for all large }n.\end{equation}
\item[(II)] All of the following holds:
\begin{itemize}
\item[(A)] $M$ is a rational matrix function.
\item[(B)] $M$ has no poles on $\partial\bbD\setminus\{\pm1\}$, and at most simple poles at $\pm1$.
\item[(C)] $(M(z)-M^\sharp(z))^{-1}$ has no poles in $\bbC\setminus\{0\}$ except at $z=\pm1$ where there might be simple poles.
\item[(D)] If $M$ has a pole at $z_j\in \bbD$ and at $z_j^{-1}$, then 
\begin{gather}
\label{p3_eq5.9}
\ran \res_{z=z_j}M(z) \subseteq \ker (M(z_j^{-1})-M^\sharp(z_j^{-1}))^{-1} ,\\
\label{p3_eq5.10}
\ran \res_{z=z_j}M(z) \subseteq \left(\ran (M(z^{-1}_j)-M^\sharp(z^{-1}_j))^{-1} M(z^{-1}_j)\right)^\perp.
\end{gather}
\end{itemize}
\end{itemize}
\end{theorem}

\begin{remarks}
1. Condition \eqref{p3_eq5.4}/\eqref{p3_eq5.9} implies that $(M(z)-M^\sharp(z))^{-1} M(z)$ is analytic at $z_j^{-1}$, so the right-hand sides of \eqref{p3_eq5.5}/\eqref{p3_eq5.10} make sense.

2.  $M$ can have poles of at most order $1$ in  $\overline{\bbD}$, however poles in $\bbC\setminus \overline{\bbD}$ may be of arbitrary order. In (D), if $M$ is assumed to have poles\textit{ of order $1$} at both $z_j$ and $z_j^{-1}$ then \eqref{p3_eq5.4}--\eqref{p3_eq5.5} (\eqref{p3_eq5.9}--\eqref{p3_eq5.10}) are equivalent to
\begin{gather*}
\ran \widetilde{w}_j \subseteq \ran (\widetilde{w}_j-z_j^2 \widetilde{q}_j),\\
\ran \widetilde{w}_j \cap \ran \widetilde{q}_j=\varnothing,
\end{gather*}
where $\widetilde{w}_j=-\res_{z=z_j} M(z)$, $\widetilde{q}_j=\res_{z= z^{-1}_j} M(z)$ (see Proposition \ref{propos}).

3. If $l=1$ then (D) is equivalent to the condition that $M$  has no simultaneous singularities at points $z_j$ and $z_j^{-1}$ (see Proposition \ref{propos}).

4. See also \cite[Thm 14]{Geronimo} for a somewhat related result on the relation between the exponential decay of Jacobi parameters and properties of the measure $\mu$ (for the scalar $l=1$ case).

5. Conditions (A) and (C) can be restated in terms of the meromorphic continuation 
of the absolutely continuous density $f(2\cos\theta)$ (as a function of $e^{i\theta}\in\partial\bbD$). 
  Condition (B) of course just means that there is no point spectrum of $\mu$ on $[-2,2]$. Condition (D) depends on both absolutely continuous and pure point parts of the measure. 
\end{remarks}

\end{subsection}


\section{Jost Asymptotics for Matrix-Valued Orthogonal Polynomials}
In this section we will be using notation
\begin{align*}
&\calJ^{(k)}=\left(
\begin{array}{cccc}
B_{k+1}&A_{k+1}&\mathbf{0}&\\
A_{k+1}^{*}&B_{k+2}&A_{k+2}&\ddots\\
\mathbf{0}&A_{k+2}^*&B_{k+3}&\ddots\\
&\ddots&\ddots&\ddots\end{array}\right),
 \\
&\widetilde{\calJ}_k=\left(
\begin{array}{cccccccc}
B_1&A_1&\mathbf{0}\\
A_1^{*}&B_2&A_2\\
\mathbf{0}&A_2^*&\ddots&\ddots\\
&&\ddots&\ddots&\ddots\\
&&&A_{k-1}^* &B_k&A_k&\mathbf{0}\\
&&&\mathbf{0}&A_k^*&\mathbf{0}&\mathbf{1}\\
&&&\mathbf{0}&\mathbf{0}&\mathbf{1}&\mathbf{0}&\ddots\\
&&&&&&\ddots&\ddots
\end{array}\right).
\end{align*}

Recall that we introduced the $M$-functions $M(z)=-\mathfrak{m}(z+z^{-1})$. Denote $M^{(k)}(z)$ to be the $M$-function corresponding to $\calJ^{(k)}$ (in particular $M^{(0)}=M$). Then the relation \eqref{m_recur} takes the form
\begin{equation}\label{p3_m-recur}
A_{n+1} M^{(n+1)}(z)A_{n+1}^*=\left(z+\frac1z\right)\mathbf{1}-B_{n+1}-{M^{(n)}(z)}^{-1}
\end{equation}
for $z\in \bbD$, $n\ge0$.

Since $M^{(n)}(z)/z=\bdone+O(z)$ at $z=0$, this gives
\begin{equation}\label{p3_m-taylor}
\left(\frac{M^{(n)}(z)}{z}\right)^{-1}=\bdone-B_{n+1}z-(A_{n+1}A_{n+1}^*-\bdone)z^2+O(z^3).
\end{equation}

\subsection{Jost Function via the Geronimo--Case Equations}

\subsubsection{Jost function for eventually free Jacobi matrices}

First we will show existence and derive some properties of the Jost solution and the Jost function for the matrices $\widetilde{\calJ}_k$. Clearly we can construct a unique solution $u_n(z;\widetilde{\calJ}_k)$ which solves~(\ref{p3_eq1}) for $\widetilde{\calJ}_k$ and satisfies $u_n(z;\widetilde{\calJ}_k)=z^n \mathbf{1}$ if $n\ge k+1$, where $z+z^{-1}=E$. 

Since $u_k(z;\widetilde{\calJ}_k)=z^k A_k^{-1}$, taking the Wronskian at $n=k$, we find,
\begin{equation*}
u(z;\widetilde{\calJ}_k)=z^k \mathfrak{p}_k^L(z+z^{-1};\widetilde{\calJ}_k)-z^{k+1}A_k^* \mathfrak{p}_{k-1}^L(z+z^{-1};\widetilde{\calJ}_k).
\end{equation*}

This suggests to define
\begin{equation}\label{p3_eq3}
g_n(z)=z^n\left( \mathfrak{p}_n^L\left(z+z^{-1}; \calJ\right)-zA_n^* \mathfrak{p}_{n-1}^L\left(z+z^{-1};\calJ\right)\right)
\end{equation}
and
\begin{equation*}\label{p3_eq4}
c_n(z)=z^n \mathfrak{p}_n^L\left(z+z^{-1};\calJ\right).
\end{equation*}

Clearly $g_n$ is a polynomial in $z$ of degree at most $2n$, and $c_n$ of degree exactly $2n$. The equation~(\ref{p3_eq3}) can be written as
\begin{equation}\label{p3_eq5}
g_{n}(z)=c_{n}(z)-z^2 A_{n}^* c_{n-1}(z).
\end{equation}

Since $\mathfrak{p}_n^L(z;\calJ)=\mathfrak{p}_n^L(z;\widetilde{\calJ}_k)$ for $n\le k$, we have
\begin{equation}\label{p3_eq6}
g_n(z)=u(z;\widetilde{\calJ}_n).
\end{equation}

Multiplying by $z^{n+1}$ the recursion relation for left orthogonal polynomials (we will start writing $\mathfrak{p}_n(z)$ instead of $\mathfrak{p}_n(z;\calJ)$ when $\calJ$ is clear from the context)
\begin{equation*}
A_{n+1}\mathfrak{p}_{n+1}^L\left(z+\frac1z\right)+\left(B_{n+1}- \left(z+\frac1z\right)\mathbf{1}\right)\mathfrak{p}_n^L\left(z+\frac1z\right)\\
+A_n^* \mathfrak{p}_{n-1}^L\left(z+\frac1z\right)=\mathbf{0}
\end{equation*}
and using~(\ref{p3_eq5}), we get
\begin{equation}\label{p3_gc1}
A_{n+1}c_{n+1}(z)=\left(z^2 \mathbf{1}-z B_{n+1}\right) c_n(z)+g_n(z).
\end{equation}
Combining (\ref{p3_eq5}) and (\ref{p3_gc1}), we obtain
\begin{equation}\label{p3_gc2}
A_{n+1}g_{n+1}(z)=\left(z^2 \left(\mathbf{1}-A_{n+1}A_{n+1}^*\right)-z B_{n+1}\right) c_n(z)+g_n(z).
\end{equation}

The recursion equations~(\ref{p3_gc1}) and~(\ref{p3_gc2}) with the initial conditions $g_0(z)=c_0(z)=\mathbf{1}$ are called the Geronimo-Case equations. They can also be written in the form
\begin{equation}\label{p3_eq8}
\left(\begin{array}{c}c_{n+1}\\g_{n+1}\end{array}\right)=V_{n+1}\left(\begin{array}{c}c_{n}\\g_{n}\end{array}\right),
\end{equation}
where $V_n$ is the $2l\times 2l$ matrix
\begin{equation}\label{p3_eq9}
V_n(z)=\left(\begin{array}{cc}A_n^{-1}&\mathbf{0}\\\mathbf{0}&A_n^{-1}\end{array}\right) \left(\begin{array}{cc}z^2 \mathbf{1}-z B_n&\mathbf{1} \\ z^2(\mathbf{1}-A_nA_n^*)-zB_n& \mathbf{1}\end{array}\right).
\end{equation}

Since $u=g_n$ if $A_k=\mathbf{1}, B_k=\mathbf{0}$ for $k\ge n+1$, it is straightforward to see the following theorem holds.

\begin{theorem}\label{p3_th1}
Let $A_kA_k^*-\mathbf{1}=B_k=\mathbf{0}$ for $k\ge n+1$ (i.e., $\calJ=\widetilde{\calJ}_n$), then $u(z;\calJ)$ is a polynomial. Moreover:
\begin{itemize}
\item if $A_nA_n^*\ne \mathbf{1}$, then $\deg(u)=2n$;
\item if $A_nA_n^*=\mathbf{1}$, but $B_n\ne\mathbf{0}$, then $\deg(u)=2n-1$.
\end{itemize}
\begin{proof}
 By~\eqref{p3_eq6}, $u(z;\calJ)=g_n(z)$, and then \eqref{p3_gc2} gives $$u(z;\calJ)=A_n^{-1} \left[\left(z^2 \left(\mathbf{1}-A_{n}A_{n}^*\right)-z B_{n}\right) c_{n-1}(z)+g_{n-1}(z)\right].$$
 Since $\deg g_{k}\le 2k$ and $\deg c_k=2k$, we obtain each statement of the theorem by induction.
\end{proof}
\end{theorem}

\medskip

\subsubsection{The general case}

Just as in \cite{DS2}, we will be making one of the three successively stronger hypotheses on the Jacobi coefficients:
$$
\sum_{n=1}^\infty \left[ ||B_n||+||\mathbf{1}-A_nA_n^*||\right]<\infty \eqno (A1)
$$
$$
\sum_{n=1}^\infty n\left[ ||B_n||+||\mathbf{1}-A_nA_n^*||\right]<\infty \eqno (A2)
$$
$$
||B_n||+||\mathbf{1}-A_nA_n^*|| \le C R^{-2n} \quad \mbox{for some } R>1 \eqno (A3)
$$
and study properties of the Jost function for each case.

Note that we have the following:
\begin{lemma}\label{p3_lm4}
If the Jacobi parameters satisfy (A1), and $\calJ$ is of type asymptotic to $1$,
then the product $\stackrel{\curvearrowright}{\prod}_{n=1}^\infty A_n$ converges, and the limit is an invertible matrix. Moreover, $\prod_{n=1}^\infty ||A_n^{-1}||<\infty$ and $\prod_{n=1}^\infty ||A_n||<\infty$, and the products converge absolutely.
\begin{proof}
Assume $\calJ$ is of type $1$, i.e., $A_n=A_n^*>0$. Then $\prod_{n=1}^\infty ||A_n^{-1}||<\infty$ follows from
\begin{equation}\label{p3_eq95}
\begin{aligned}
\sum_{n=1}^\infty |1-||A_n^{-1}||| & \le \sum_{n=1}^\infty ||\mathbf{1}-A_n^{-1}||\le \sum_{n=1}^\infty ||A_n^{-1}||\,||\mathbf{1}-A_n||\\
&\le \sup_j||A_j^{-1}|| \sum_{n=1}^\infty ||\mathbf{1}-A_n^2|| \, ||(\mathbf{1}+A_n)^{-1}|| \\
&\le c \sum_{n=1}^\infty ||\mathbf{1}-A_n^2||<\infty,
\end{aligned}
\end{equation}
where we can bound $||A_n^{-1}||$ and $||(\mathbf{1}+A_n)^{-1}||$ uniformly since $\calJ$ is in the Nevai class, so $A_n\to\mathbf{1}$, so $(\mathbf{1}+A_n)^{-1}\to\frac12 \mathbf{1}$.

The bound for $\sum_{n=1}^\infty |1-\norm{A_n}|$ is analogous.

Note that we also showed that $\sum_{n=1}^\infty ||\bdone-A_n||<\infty$. It is proven in \cite{Trench} that given this, the limit  $\stackrel{\curvearrowright}{\prod}_{n=1}^\infty A_n$ exists and is invertible.

Now let $\widetilde{\calJ}$ be any matrix satisfying (A1) asymptotic to type $1$, satisfying \eqref{p1_eq2.3}. Then $$\stackrel{\curvearrowright}{\prod}_{n=1}^N \widetilde{A}_n=\stackrel{\curvearrowright}{\prod}_{n=1}^N A_n \sigma_{N+1}$$ also has an invertible limit.
\end{proof}
\end{lemma}

Define $g_n$ and $c_n$ by~\eqref{p3_gc1} and \eqref{p3_gc2} with the initial conditions $g_0(z)=c_0(z)=\mathbf{1}$.

\begin{lemma}\label{p3_lm5} Assume $\calJ$ is of type $1$.
\begin{itemize}
\item[(i)] Let $(A1)$ hold. Then uniformly on compacts $K$ of $\overline{\bbD} \setminus \{\pm 1\}\equiv \bbE$,
\begin{equation}\label{p3_eq10}
\sup_{n\in\bbN, z\in K}||c_n(z)||+||g_n(z)|| < \infty.
\end{equation}
\item[(ii)] Let $(A2)$ hold. Then
\begin{align}\label{p3_eq101}
\sup_{n\in\bbN, z\in \overline{\bbD}}||g_n(z)|| < \infty,
\\ \label{p3_eq102}
\sup_{n\in\bbN, z\in \overline{\bbD}}\frac{||c_n(z)||}{1+n} < \infty.
\end{align}
\item[(iii)] Let $(A3)$ hold.
Let $K$ be  any compact subset of $z\in\{z\mid |z|<R\}\equiv \bbD_R$ with $r=\sup_{z\in K} |z|>1$. There exists some constant $C$ such that for all $z\in K$
\begin{equation}\label{p3_eq11}
||c_n(z)||+||g_n(z)||\le C\left[\max(1,r)\right]^{2n}.
\end{equation}
\end{itemize}
\medskip
\noindent In each of these cases the limit
\begin{equation*}
g_\infty(z)=\lim_{n\to\infty} g_n(z)
\end{equation*}
exists, uniformly on compacts of the corresponding region: $\bbE$ for $(A1)$, $\overline{\bbD}$ for $(A2)$, and $\bbD_R$ for $(A3)$. $g_\infty$ is continuous there, and analytic in the interior.
%
\begin{proof}
(i) Define the norm $\left|\left|\left(\begin{array}{c}A\\B\end{array}\right)\right|\right|_{2l\times l}=||A||+||B||$ for any $l\times l$ matrices $A,B$. For any $2l\times 2l$ matrix $V$ let $||V||_{in}$   be the corresponding induced operator norm. Taking~(\ref{p3_eq8}) into account, the estimates~(\ref{p3_eq10}) and~(\ref{p3_eq11}) will be proved if we show the corresponding results for $||V_n(z)\ldots V_1(z)||_{in}$. Observe that for $z\ne\pm 1$,
$$
\left(\begin{array}{cc}z^2 \mathbf{1}&\mathbf{1}\\\mathbf{0}&\mathbf{1}\end{array}\right) =L(z) \left(\begin{array}{cc}z^2 \mathbf{1}&\mathbf{0}\\\mathbf{0}&\mathbf{1}\end{array}\right) L(z)^{-1}, $$
where
$$
L(z)=\left(\begin{array}{cc}\mathbf{1}&\frac{1}{1-z^2}\mathbf{1}\\\mathbf{0}&\mathbf{1}\end{array}\right) ,\quad L(z)^{-1}= \left(\begin{array}{cc}\mathbf{1}&-\frac{1}{1-z^2}\mathbf{1}\\\mathbf{0}&\mathbf{1}\end{array}\right).$$
So denoting
$$
F_n=L(z)^{-1}
\left(\begin{array}{cc}-z B_n&\mathbf{0} \\ z^2(\mathbf{1}-A_nA_n^*)-zB_n& \mathbf{0}\end{array}\right)
L(z),
$$
we obtain from~(\ref{p3_eq9}),
\begin{multline*}
V_n=\left(\begin{array}{cc}A_n^{-1}&\mathbf{0}\\\mathbf{0}&A_n^{-1}\end{array}\right) L(z) \left[\left(\begin{array}{cc}z^2\mathbf{1}&\mathbf{0}\\\mathbf{0}&\mathbf{1}\end{array}\right)+F_n\right] L(z)^{-1} \\=L(z)
\left(\begin{array}{cc}A_n^{-1}&\mathbf{0}\\\mathbf{0}&A_n^{-1}\end{array}\right) \left[\left(\begin{array}{cc}z^2\mathbf{1}&\mathbf{0}\\\mathbf{0}&\mathbf{1}\end{array}\right)+F_n\right] L(z)^{-1}
\end{multline*}
since $L(z)$ and $\left(\begin{array}{cc}A_n^{-1}&\mathbf{0}\\\mathbf{0}&A_n^{-1}\end{array}\right)$ commute.

Then we get that for any $z$, $z\ne\pm1$,
\begin{multline}\label{p3_eq12}
\norm{V_n\ldots V_1 }_{in}
\le \norm{L(z)}_{in} \, \norm{L(z)^{-1}}_{in}  \left[\max(1,|z|)\right]^{2n} \\
\times \prod_{j=1}^n ||A_j^{-1}||  \prod_{j=1}^n\left(1+\norm{L(z)}_{in} \, \norm{L(z)^{-1}}_{in} \left(||B_j||+||\mathbf{1}-A_jA_j^*||\right)\right). \end{multline}
By Lemma~\ref{p3_lm4}, we can bound $\prod_{j=1}^n ||A_j^{-1}||$.
\smallskip

For any compact $K$ of $\bbE$, $\sup_{z\in K} \norm{L(z)}_{in} \, \norm{L(z)^{-1}}_{in}<\infty$, so taking supremum in~\eqref{p3_eq12} over $z\in K$ and using (A1) we obtain
\begin{equation*}
\sup_{n\in\bbN, z\in K}||c_n(z)||+||g_n(z)|| = M< \infty
\end{equation*}
for some constant $M$.

\smallskip

(ii) Note that  by Lemma~\ref{p3_lm3}(i), we have
\begin{equation*}
\sup_{\Lambda\subset\bbN} \prod_{j\in\Lambda} ||A_j^{-1}||=p<\infty.
\end{equation*}

Let us show inductively that
\begin{equation*}
||g_n(z)|| \le \prod_{j=1}^n ||A_j^{-1}||\, \prod_{j=1}^n \left[1+j ( ||B_j||+||\mathbf{1}-A_jA_j^*||)\right]
\end{equation*}
and
\begin{equation*}
||c_n(z)|| \le (n+1)\,\prod_{j=1}^n ||A_j^{-1}|| \, \prod_{j=1}^n \left[1+j ( ||B_j||+||\mathbf{1}-A_jA_j^*||)\right].
\end{equation*}

For $n=0$ the inequalities are trivial. Now, if these inequalities hold for $n$ then using~\eqref{p3_gc1} and \eqref{p3_gc2}:
\begin{multline*}
||g_{n+1}(z)||  \le ||A_{n+1}^{-1}||\left[ (n+1) (||B_{n+1}||+||\mathbf{1}-A_{n+1}A_{n+1}^*||)+1 \right] \\
\times \prod_{j=1}^n ||A_j^{-1}||\, \prod_{j=1}^n \left[1+j ( ||B_j||+||\mathbf{1}-A_jA_j^*||)\right]
\end{multline*}
and
\begin{equation*}
\begin{aligned}
||c_{n+1}(z)|| &\le ||A_{n+1}^{-1}|| \left[ (n+1) (1+||B_{n+1}||)
+1 \right]
\\ & \qquad\quad\qquad \times \prod_{j=1}^n ||A_j^{-1}||\, \prod_{j=1}^n \left[1+j ( ||B_j||+||\mathbf{1}-A_jA_j^*||)\right] \\
&\le (n+2)\,\prod_{j=1}^{n+1} ||A_j^{-1}|| \, \prod_{j=1}^{n+1} \left[1+j ( ||B_j||+||\mathbf{1}-A_jA_j^*||)\right].
\end{aligned}
\end{equation*}
By Lemma~\ref{p3_lm4}, $\prod_{n=1}^\infty ||A_n^{-1}||$ is absolutely convergent, so $(A2)$ implies \eqref{p3_eq101} and \eqref{p3_eq102}.

\smallskip

(iii)  Since $||g_n||$ and $||c_n||$ are subharmonic functions, by the maximum principle we need to prove the estimate~\eqref{p3_eq11} for the circle $|z|=r$. This follows immediately from~\eqref{p3_eq12}. Note that this property does not really require $(A3)$, just $(A1)$ (the existence of the limit however will).

\medskip

Now to show the convergence of $g_n$, note that by~(\ref{p3_gc2}),
\begin{multline}\label{p3_eq14}
||g_{n+1}(z)-g_n(z)||=||A_{n+1}^{-1}\left(z^2 \left(\mathbf{1}-A_{n+1}A_{n+1}^*\right)-z B_{n+1}\right) c_n(z)
\\
+\left(A_{n+1}^{-1}-\mathbf{1}\right)g_n(z)|| \\ \le \left[\sup_j||A_j^{-1}|| \, \left[\max(1,r)\right]^{2}\, \left( ||B_n||+||\mathbf{1}-A_nA_n^*||\right) + ||\mathbf{1}-A_{n+1}^{-1}||\right] \\
\times \sup_{n\in\bbN, z\in K}\left(||c_n(z)||+||g_n(z)||\right).
\end{multline}
Since we are in the type 1 situation, we can use the same reasoning as in~(\ref{p3_eq95}) to get $||\mathbf{1}-A_{n+1}^{-1}||\le c \norm{\mathbf{1}-A_{n+1}A_{n+1}^*}$, and then~(\ref{p3_eq14}), together with the estimates in (i), (ii), and (iii), gives
$\sum_{n=0}^\infty ||g_{n+1}(z)-g_n(z)||<\infty$ uniformly on compacts of $\bbE$, $\overline{\bbD}$, $\bbD_R$, respectively. This proves the existence and analyticity/continuity properties of $g_\infty$.
\end{proof}
\end{lemma}

As a consequence we obtain Szeg\H{o} asymptotics of the orthonormal polynomials in the unit disk (compare with \cite{K_sz}).
\begin{theorem} \label{p3_th2}
Assume $(A1)$ holds, i.e., $\sum_{n=1}^\infty \left[ ||B_n||+||\mathbf{1}-A_nA_n^*||\right]<\infty$, and let $\calJ$ be of type $1$. Then uniformly on compacts of $\bbD$ the limit
\begin{equation}\label{p3_eq3.26}
\lim_{n\to\infty} z^n \mathfrak{p}_n^L\left(z+z^{-1}\right)
\end{equation}
exists, and is equal to $\frac{1}{1-z^2}\,g_\infty(z)$.
\begin{proof}
Note that by Lemma~\ref{p3_lm4}, $\prod_{n=1}^\infty ||A_n^{-1}||$ is absolutely convergent, so by Lemma \ref{p3_lm3}(i), we have
\begin{equation*}
\sup_{\Lambda\subset\bbN} \prod_{j\in\Lambda} ||A_j^{-1}||=p<\infty.
\end{equation*}

Let $K$ be any compact of $\bbD$, and $M=\sup_{n\in\bbN, z\in K}||c_n(z)||+||g_n(z)||$. By the Geronimo--Case equations,
\begin{equation*}
||c_n-A_{n}^{-1} g_{n-1}-z^2 A_n^{-1} c_{n-1}||\le M \norm{A_n^{-1}} \cdot \norm {B_{n}}\le Mp \, \norm{B_n}.
\end{equation*}
Repeating this, we get
\begin{equation*}\label{p3_eq2.30}
\begin{aligned}
||c_n-A_{n}^{-1} g_{n-1}-&z^2 A_n^{-1} A_{n-1}^{-1} g_{n-2}- z^4 A_n^{-1}A_{n-1}^{-1} c_{n-2}|| \\
&\le M p \,  \norm {B_{n}}+|z|^2 M \norm{A_n^{-1}} \cdot \norm{A_{n-1}^{-1}} \cdot \norm {B_{n-1}} \\ &\le M p \, \norm {B_{n}}+|z|^2 M p \, \norm {B_{n-1}}.
\end{aligned}
\end{equation*}
Iterating it further, we get
\begin{equation}\label{p3_eq2.31}
||c_n-f_n|| \le M p \sum_{j=1}^n |z|^{2(n-j) }\norm {B_{j}},
\end{equation}
where
\begin{multline*}
f_n=A_n^{-1}g_{n-1}+ z^2 A_n^{-1} A_{n-1}^{-1} g_{n-2}+\ldots+z^{2(n-1)} A_n^{-1} A_{n-1}^{-1}\ldots A_1^{-1} g_0 \\
+ z^{2n} A_n^{-1}A_{n-1}^{-1}\ldots A_1^{-1} c_{0}.
\end{multline*}
By Lemma~\ref{p3_lm3}(ii) the right-hand side of~\eqref{p3_eq2.31} goes to zero
. Finally, note that
\begin{equation}\label{p3_eq2.33}
\begin{aligned}
& \left|\left|\stackrel{\curvearrowright}{\prod_{k=n+1}^\infty}  A_{k} \, g_\infty \frac{1-z^{2n}}{1-z^2}-f_n\right|\right| \\
&\le p \sum_{j=0}^{n-1} |z|^{2(n-1-j)} \left|\left|\stackrel{\curvearrowright}{\prod_{k=n+1}^\infty} A_{k} \, g_\infty-A_n^{-1} A_{n-1}^{-1}\ldots A_{j+1}^{-1} g_j\right|\right| \\
&\le
p^2 \sum_{j=0}^{n-1} |z|^{2(n-1-j)} \left|\left|\stackrel{\curvearrowright}{\prod_{k=1}^\infty} A_{k} \, g_\infty-A_{1}\ldots A_{j} g_j\right|\right|.
\end{aligned}
\end{equation}
By Lemma~\ref{p3_lm4}, the product $\stackrel{\curvearrowright}{\prod_{k=1}^\infty} A_{k}$ converges, and by Lemma~\ref{p3_lm3}(ii) the right-hand side of \eqref{p3_eq2.33} goes to zero. Easy to see that the convergence in \eqref{p3_eq2.31} and \eqref{p3_eq2.33} is actually uniform. Thus we established $\lim_{n\to\infty} c_n=\frac{1}{1-z^2} g_\infty$.
\end{proof}
\end{theorem}

\begin{remark}
Another way of showing this is to use the analogous arguments to~\cite[Lemma 3.7.5]{Rice} to show that Szeg\H{o} asymptotics (i.e., \eqref{p3_eq3.26}) at $z\in\bbD$ holds  if and only if the Jost asymptotics does (i.e., \eqref{p3_eq21}), so that Theorem~\ref{p3_th5} implies Theorem~\ref{p3_th2}.
\end{remark}

%
%

Denote the limit function $g_\infty(z)$ of Lemma~\ref{p3_lm5} as $u(z;\calJ)$ and call it the Jost function (in Theorem \ref{p3_th5} below we will show that this indeed agrees with our earlier Definition \ref{def3}). 
Lemma~\ref{p3_lm5} establishes the existence of the Jost function for the type $1$ situation only. The next theorem says that the Jost function exists if and only if the Jacobi matrix is asymptotic to type $1$. Note that by Lemma \ref{p1_thm2}, type 2 and type 3 are asymptotic to type 1 under the condition (A1).

\begin{theorem}\label{p3_th3}
Let $\calJ$ with Jacobi parameters $(A_n)_{n=1}^\infty$, $(B_n)_{n=1}^\infty$ be of type $1$ and satisfy $(A1)$. Let $\widetilde{\calJ}$ with Jacobi parameters $(\widetilde{A}_n)_{n=1}^\infty$, $(\widetilde{B}_n)_{n=1}^\infty$ be equivalent to $\calJ$, i.e,
\begin{align*}
\widetilde{A}_n=\sigma_n^* A_n \sigma_{n+1},\\
\widetilde{B}_n=\sigma_n^* B_n \sigma_n
\end{align*}
for some unitary $\mathbf{1}=\sigma_1, \sigma_2, \sigma_3,\ldots$
Then the Jost function for $\widetilde{\calJ}$ exists if and only if $\lim_{n\to\infty} \sigma_n$ exists, in which case
\begin{equation*}\label{p3_eq17}
u(z;\widetilde{\calJ})=\lim_{n\to \infty}\sigma_n^* \, u(z;\calJ)\sigma_1.
\end{equation*}
\end{theorem}
\begin{proof}
We prove inductively that $\widetilde{g}_n=\sigma_{n+1}^* g_n\sigma_1$ and $\widetilde{c}_n=\sigma_{n+1}^* c_n\sigma_1$. For $n=0$ this is trivial, and assuming this holds for $n$, we prove it for $n+1$:
\begin{equation*}
\begin{aligned}
\widetilde{g}_{n+1}(z)&=\widetilde{A}_{n+1}^{-1}\left[\widetilde{g}_n(z)+\left(z^2 \left(\mathbf{1}-\widetilde{A}_{n+1}\widetilde{A}_{n+1}^*\right)-z \widetilde{B}_{n+1}\right) \widetilde{c}_n(z)\right]\\
&=\sigma_{n+2}^* A_{n+1}^{-1}\sigma_{n+1}\left[\sigma_{n+1}^* g_n(z) \right.\\
&\qquad\qquad \left.+\left(z^2 \left(\mathbf{1}-\sigma_{n+1}^* A_{n+1}A_{n+1}^*\sigma_{n+1}\right)-z \sigma_{n+1}^* B_{n+1}\sigma_{n+1}\right) \sigma_{n+1}^* c_n(z)\right]\sigma_1\\
&=\sigma_{n+2}^* A_{n+1}^{-1}\left[\left(z^2 \left(\mathbf{1}-A_{n+1}A_{n+1}^*\right)-z B_{n+1}\right) c_n(z)+ g_n(z)\right]\sigma_1 \\
&=\sigma_{n+2}^* g_{n+1}(z)\sigma_1,
\end{aligned}
\end{equation*}
and similarly for $\widetilde{c}_{n+1}=\sigma_{n+2}^* c_{n+1}\sigma_1$. The limit $\lim_{n\to\infty} g_n(z)$ exists by Lemma~\ref{p3_lm5}, so $\lim_{n\to\infty} \widetilde{g}_n(z)$ exists if and only if  the limit $\lim_{n\to\infty} \sigma_n$ exists, in which case $u(z;\widetilde{\calJ})=\lim_{n\to \infty}\sigma_n^* \, u(z;\calJ)\sigma_1$.
\end{proof}

Assume $\calJ$ is a Jacobi matrix asymptotic to type 1, and let its Jacobi parameters satisfy~$(A1)$, $(A2)$, or $(A3)$. Then so do the parameters of $\calJ^{(k)}$ for all $k$, and thus $u(z;\calJ^{(k)})$ exists in $\bbE$, $\overline\bbD$, $\bbD_R$, respectively (which will be called ``the appropriate region'' in what follows). We define the Jost solution (in Theorem \ref{p3_th5} below we will show it is indeed the Jost solution we defined earlier in Definition \ref{def2}) by
\begin{equation}\label{p3_eq19}
u_n(z;\calJ)=z^n u(z;\calJ^{(n)})A_n^{-1}.
\end{equation}

Observe that by (the arguments of) Theorem~{\ref{p3_th3}}, the Jost solutions of equivalent Jacobi matrices are related via
\begin{equation*}
u_k(z;\widetilde{\calJ})=\lim_{n\to \infty}\sigma_n^* \, u_k(z;\calJ)\sigma_k.
\end{equation*}

Recall that $\mathfrak{m}(z)=\int \frac{1}{x-z}d\mu(x)$ and $M(z)=-\mathfrak{m}(z+z^{-1};\calJ)$. For each discrete eigenvalue $E_j$ of $\calJ$ outside $[-2,2]$, let $z_j\in\bbD\cap\bbR$ be such that $z_j+{z_j}^{-1}=E_j$, and denote $\widetilde{w}_j=-\lim_{z\to z_j} (z-z_j) M(z)$, $w_j=\mu(E_j)=-\lim_{E\to E_j} (E-E_j) \mathfrak{m}(E)=(z_j^{-1}-z_j)z_j^{-1} \widetilde{w}_j$ ($w_j,\widetilde{w}_j\ge\bdnot$).

Recall that $f^\sharp(z)=f(\bar{z}^{-1})^*$, and that zeros of a matrix-valued function $f$ are defined to be the poles of its inverse $f^{-1}$. A zero of $f$ is simple if the corresponding pole of $f^{-1}$ is simple.


\begin{theorem}\label{p3_th5}
Assume $\calJ$ is a Jacobi matrix asymptotic to type $1$, and let its Jacobi parameters satisfy $(A1)$, $(A2)$, or $(A3)$.
\begin{itemize}
\item[(i)] $u_n(z;\calJ)$ in the appropriate region ($\bbE$, $\overline\bbD$, $\bbD_R$, resp.) satisfies
\begin{equation}\label{p3_eq20}
u_{n+1}(z;\calJ)A_{n}^*+u_n(z;\calJ)(B_{n}-(z+z^{-1}) \mathbf{1})+u_{n-1}(z;\calJ)A_{n-1}=\mathbf{0}, \quad n=1,2,\ldots.
\end{equation}
\item[(ii)] In the appropriate region,
\begin{equation}\label{p3_eq21}
\lim_{n\to\infty} z^{-n} u_n(z;\calJ)=\mathbf{1}.
\end{equation}
\item[(iii)] For $z\in \D$,
\begin{equation}\label{p3_eq22}
u(z;\calJ^{(1)})=z^{-1}u(z;\calJ) M(z;\calJ) A_1.
\end{equation}
\item[(iv)] The only zeros of $u(z;\calJ)$ in $\D$ are at the real points $z_j$ with $z_j+z_j^{-1}\equiv E_j$ a discrete eigenvalue of $\calJ$. Each zero of $u(z;\calJ)$ in $\D$ is simple
    , and the order of $z_j$ as a zero of $\det u(z;\calJ)$ equals to the multiplicity of $E_j$ as an eigenvalue of $\calJ$. Moreover,
    \begin{equation}\label{p3_eq23}
    \ker u(z_j;\calJ)=\ran w_j=\ran \widetilde{w}_j.
    \end{equation}
\item[(v)] The only zeros of $u(z;\calJ)$ in $\partial\D$ are possible ones at $\pm1$, in which case they are simple.
\item[(vi)] $M(z;\calJ)$ has a continuation from $\D$ to $\overline{\D}\setminus\{\pm1\}$, which is everywhere finite and invertible on $\partial\D\setminus\{\pm1\}$, and
    \begin{equation}\label{p3_eq24}
    \imm M(e^{i\theta})=\sin\theta \left[u(e^{i\theta};\calJ)^* u(e^{i\theta};\calJ)\right]^{-1}.
    \end{equation}
\item[(vii)] The following recurrence holds:
\begin{equation*}
u(z;\calJ^{(2)})
=z^{-1} u(z;\calJ^{(1)}) A_1^{-1}((z+z^{-1})\bdone-B_1) {A_1^*}^{-1} A_2-
z^{-2} u(z;\calJ)  {A_1^*}^{-1} A_2.
\end{equation*}
\end{itemize}
Now assume $(A3)$ holds.
\begin{itemize}
\item[(viii)] $M$ can be extended meromorphically to $\{z\mid |z|<R\}$, and
\begin{equation}\label{p3_eq25}
M(z)=M^\sharp(z)+(z-z^{-1})\left[u^\sharp(z;\calJ)u(z;\calJ)\right]^{-1}, \quad R^{-1}<|z|<R.
\end{equation}
\item[(ix)]  For each $z_j$ with $R^{-1}<|z_j|<1$,
\begin{equation}\label{p3_eq3.42}
 \widetilde{w}_j u(1/{\bar{z}_j};\calJ)^*=-(z_j-z_j^{-1}) \res_{z=z_j} u(z;\calJ)^{-1},
 \end{equation}
in particular,
\begin{equation}\label{p3_eq3.43}
\ker u(1/{\bar{z}_j};\calJ)^*\subseteq \ker \res_{z=z_j} u(z;\calJ)^{-1}=\ran u(z_j;\calJ).
\end{equation}
\end{itemize}
\begin{remarks}
1. Part (vi) shows that if $(A1)$ holds then there is no point spectrum of $\calJ$ in $[-2,2]$.

2. Part (iv), together with analyticity of $u$, shows that under $(A3)$ there exist at most finitely many eigenvalues of $\calJ$ outside of $[-2,2]$. One could show that under $(A2)$ the same result holds (see \cite{Geronimo_matrix}). Under $(A1)$ there may be infinitely many eigenvalues outside of $[-2,2]$ with $-2$ and $2$ as their only accumulation points.

3. Part (vii) shows that if $u(z;\calJ)$ and $u(z;\calJ^{(1)})$ are analytic, then so is $u(z;\calJ^{(n)})$ for any $n$. This is  why the inductive argument for the inverse direction works.
\end{remarks}
\begin{proof}
(i) Note that since $\widetilde{u}(z;\calJ_l)=g_l(z;\calJ)\to u(x;\calJ)$, it suffices to show~\eqref{p3_eq20} for $\calJ\equiv \widetilde{\calJ}_l$.

Let $v_n(z;\widetilde{\calJ}_l)$ be the ``old'' definition of Jost solution, i.e., the solution of~\eqref{p3_eq20} for $\calJ\equiv \widetilde{\calJ}_l$ such that $v_n(z;\widetilde{\calJ}_l)=z^n$ for large $n$. Note that by \eqref{p3_eq6} $v_0(z;\widetilde{\calJ}_l)=g_l(z;\widetilde{\calJ}_l)=\lim_{k\to\infty}g_k(z;\widetilde{\calJ}_l)=u_0(z;\widetilde{\calJ}_l)$, where the middle equality comes from $\eqref{p3_gc2}$.

Since $\calJ^{(k)}$ shifts indices by $k$, and $z^n=z^{-k}(z^{n+k})$, we have for all $n\ge1$ and $k\ge1$, $$v_n\left(z;\left[\widetilde{\calJ}_l\right]^{(k)}\right)=z^{-k} v_{n+k}\left(z;\widetilde{\calJ}_l\right).$$

For $n=0$, the difference equation~\eqref{p3_eq20} then gives
$$v_0\left(z;\left[\widetilde{\calJ}_l\right]^{(k)}\right)=z^{-k} v_{k}\left(z;\widetilde{\calJ}_l\right)A_k,$$ and so
$$
v_{k}\left(z;\widetilde{\calJ}_l\right)=z^k v_0\left(z;\left[\widetilde{\calJ}_l\right]^{(k)}\right) A_k^{-1}=z^k u_0\left(z;\left[\widetilde{\calJ}_l\right]^{(k)}\right) A_k^{-1}\equiv u_{k}\left(z;\widetilde{\calJ}_l\right).
$$

\medskip

(ii) It follows from~(\ref{p3_eq14}) that
\begin{equation}\label{p3_eq26}
\begin{aligned}
||u(z;\calJ^{(n)})-\mathbf{1}|| &\le \sum_{j=0}^\infty ||g_{j+1}(z;\calJ^{(n)})-g_{j}(z;\calJ^{(n)})|| \\
&\le  \sup_{k\in\bbN, z\in K}\left(||c_k(z;\calJ^{(n)})||+||g_k(z;\calJ^{(n)})||\right) \\
& \quad\times \sum_{j=n+1}^\infty \left[ \sup_k||A_k^{-1}|| \left[\max(1,r)\right]^{2}\left( ||B_j||+||\mathbf{1}-A_jA_j^*||\right) + ||\mathbf{1}-A_{j+1}^{-1}||\right].
\end{aligned}
\end{equation}

Now, assuming $\calJ$ is of type 1, we can bound $||\mathbf{1}-A_{j}^{-1}||\le c ||\mathbf{1}-A_{j}A_{j}^*||$, and then Lemma~\ref{p3_lm5} gives the convergence of the right hand side of~(\ref{p3_eq26}).

If $\widetilde{\calJ}$ is of type asymptotic to 1, then by Theorem~\ref{p3_th3} we get $$\lim_{k\to\infty}z^{-k} u_k(z;\widetilde{\calJ})=\lim_{k\to\infty}\lim_{n\to \infty}\sigma_n^* \, z^{-k}u_k(z;\calJ)\sigma_k=\lim_{n\to \infty}\sigma_n^*\lim_{k\to \infty}\sigma_k=\mathbf{1}.$$

\medskip

(iii) By \cite[Thm 2.16(iii)]{DPS}, we get $u_1(z;\calJ)=-u_0(z;\calJ)\mathfrak{m}(z+z^{-1};\calJ)$, hence
\begin{equation*}
u(z;\calJ^{(1)})=z^{-1}u_1(z;\calJ)A_1=z^{-1}u(z;\calJ)M(z;\calJ)A_1.
\end{equation*}

\medskip

(iv) Observe that if $M(z;\calJ)$ is regular at $z$, then $u(z;\calJ)$ is invertible at $z$. Otherwise we can pick an eigenvector $f$ with $f^* u(z;\calJ)=\mathbf{0}$ and see that $f^* u_1(z;\calJ)=f^* u(z;\calJ) M(z;\calJ)=\mathbf{0}$, and then $f^* u_n(z;\calJ)=\mathbf{0}$ for all $n$ from~(\ref{p3_eq20}). This would contradict (ii).

Thus the only possible zeros are at $z_j$'s with $z_j+z_j^{-1}=E_j$ being an eigenvalue of $\calJ$. Let $q_k$ be the multiplicity of $E_j$ as an eigenvalue of $\calJ^{(k)}$. By Lemma \ref{stripping2}, $\sigma(\calJ^{(N)})\subset \left[-2-\epsilon,2+\epsilon\right]$ for sufficiently big $N$, so $q_n=0$ for all $n\ge N$. Since $q_N=0$, $M(z;\calJ^{(N)})$ is regular at $z_j$, and then the arguments above show that $u(z;\calJ^{(N)})$ is invertible at $z_j$. Now let us prove the statement about zeros of the determinant inductively assuming we know it for $N, N-1,\ldots,n+1$. By \cite[Thm 2.28]{DPS}, $\det M(z;\calJ^{(n)})$ has zero of order $q_{n+1}-q_n$ at $z=z_j$, and then~(\ref{p3_eq22}) gives $\det u(z;\calJ^{(n)})=z^n \det u(z;\calJ^{(n+1)})\det M(z;\calJ^{(n)})^{-1}\det A_{n+1}^{-1}$ has zero of order $q_{n+1}-(q_{n+1}-q_n)=q_n$ at $z=z_j$. Thus $\det u(z;\calJ)$ has zero of order $q_0$ at $z=z_j$.

Hence $\dim\ker u(z_j;\calJ)\le q_0$. However, $$\mathbf{0}=\lim_{z\to z_j}(z-z_j)u(z;\calJ^{(1)})=z_j^{-1}u(z_j;\calJ)\lim_{z\to z_j}(z-z_j)M(z;\calJ)A_1=z_j^{-1}u(z_j;\calJ)\widetilde{w}_j A_1,$$ which implies $\ran \widetilde{w}_j\subseteq \ker u(z_j;\calJ)$. Then $q_0=\dim \ran \widetilde{w}_j \le\dim\ker u(z_j;\calJ)\le q_0$, which means $\ran \widetilde{w}_j= \ker u(z_j;\calJ)$. $\ran \widetilde{w}_j=\ran w_j$ is obvious.

Since $\dim\ker u(z_j;\calJ)= q_0$ and $\det u(z;\calJ)$ has zero of order $q_0$ at $z=z_j$, by Lemma~\ref{p4_lm0}  the order of the pole of $u(z;\calJ)^{-1}$ at $z=z_j$ cannot be larger than $1$.

\medskip

(v) If $z\in\partial \D$, then $u_n(z;\calJ)$ and $u_n(z^{-1};\calJ)$ solve the same Jacobi equation, and so the Wronskian $W_n(u_\cdot(z;\calJ);u_\cdot(\bar{z}^{-1};\calJ)^*)$ is constant. By (ii), the Wronskian at infinity is $$\lim_{n\to\infty} u_n(z)A_n u_{n+1}(z)^*-u_{n+1}(z)A_n^* u_n(z)^*=(z^{-1}-z)\mathbf{1},$$ while evaluating it at zero gives
\begin{equation*}
u_0(z)u_1(z)^*-u_1(z)u_0(z)^*=(z^{-1}-z)\mathbf{1},
\end{equation*}
or
\begin{equation}\label{p3_eq29}
\imm \left[u_1(e^{i\theta})u_0(e^{i\theta})^*\right]=\sin\theta\,\mathbf{1}.
\end{equation}
This implies that for $\theta \ne 0$, $u_0(e^{i\theta};\calJ)$ is invertible.

To prove that the poles at $\pm1$ are at most of order $1$, just note that using~(\ref{p3_eq24}) (which is proven in (vi)), the absolutely continuous part of $\mu$ is
\begin{equation*}
f(2\cos\theta)=\pi^{-1} \left| \imm M(e^{i\theta}) \right|=\pi^{-1}\left|\sin\theta\right|\,\left[u(e^{i\theta})^* u(e^{i\theta})\right]^{-1},
\end{equation*}
and then in order for
\begin{equation*}
\int_{-2}^2 f(x)dx=2\int_{0}^\pi \sin\theta \,f(2\cos\theta)d\theta=\frac2\pi \int_{0}^\pi \sin^2\theta\,\left[u(e^{i\theta})^* u(e^{i\theta})\right]^{-1}d\theta
\end{equation*}
to be finite, we must have that the pole of $u(z)^{-1}$ at $\pm1$ is at most of order $1$.

\medskip

(vi) By $u_1(z;\calJ)=u(z;\calJ) M(z;\calJ)$, for $\theta\ne0$,
\begin{multline*}\label{p3_eq32}
\imm M(e^{i\theta})=\imm u(e^{i\theta})^{-1}u_1(e^{i\theta})= \imm \left(u(e^{i\theta})^{-1}u_1(e^{i\theta}) \left[u(e^{i\theta})^* {u(e^{i\theta})^*}^{-1} \right]\right) \\= u(e^{i\theta})^{-1}\imm \left[u_1(e^{i\theta}) u(e^{i\theta})^* \right] {u(e^{i\theta})^*}^{-1}  =\sin\theta\,\left[u(e^{i\theta})^* u(e^{i\theta})\right]^{-1}
\end{multline*}
by~(\ref{p3_eq29}).

\medskip

(vii) This part follows immediately from \eqref{p3_eq19} and (i). One can also obtain this using (iii) and \eqref{p3_m-recur} only.

\medskip

(viii) 
By (iii), $M$ is meromorphic in the region where $u$'s are analytic. Note that \eqref{p3_eq25}  at $z=e^{i\theta}$ is  \eqref{p3_eq24}. Thus if we define $\widehat{M}(z)=M^\sharp(z)+(z-z^{-1})\left[u^\sharp(z;\calJ)u(z;\calJ)\right]^{-1}$ for $1<|z|<R$, then $M(z)=\widehat{M}(z)$ on $\partial\bbD$, and \eqref{p3_eq25} follows by analytic continuation.

\medskip

(ix) Note that $\calJ^{(1)}$ also satisfies $(A3)$, and so $u(z;\calJ^{(1)})$ is  analytic in $\bbD_R$. Combining \eqref{p3_eq22} and \eqref{p3_eq25} we obtain
\begin{equation*}
u(z;\calJ^{(1)})=z^{-1}u(z;\calJ)\left[M^\sharp(z)+(z-z^{-1})\left[u^\sharp(z;\calJ)u(z;\calJ)\right]^{-1}\right] A_1, \quad R^{-1}<|z|<R.
\end{equation*}
Analyticity of $u(z;\calJ^{(1)})$ at $z_j^{-1}$ means that the residues must cancel out:
\begin{equation*}
\begin{aligned}
\mathbf{0}&=\lim_{z\to z_j^{-1}} (z-z_j^{-1})u(z;\calJ)M^\sharp(z)+\lim_{z\to z_j^{-1}} (z-z_j^{-1})(z-z^{-1})\left[u^\sharp(z;\calJ)\right]^{-1} \\
&= u(z_j^{-1};\calJ)\lim_{z\to z_j} (z^{-1}-z_j^{-1})M(\bar{z})^*+(z_j^{-1}-z_j)\lim_{z\to z_j} (z^{-1}-z_j^{-1})\left[u(\bar{z};\calJ)^*\right]^{-1}\\
&=\frac{1}{z_j^2}u(z_j^{-1};\calJ)\widetilde{w}_j^*+\frac{1}{z_j^2}(z_j-z_j^{-1})[\lim_{z\to z_j}(z-z_j)u(z;\calJ)^{-1}]^*,
\end{aligned}
\end{equation*}
which gives~(\ref{p3_eq3.42}).

The rightmost equality of \eqref{p3_eq3.43} comes from Lemma~\ref{p4_lm0}. The inclusion part of \eqref{p3_eq3.43} follows immediately from \eqref{p3_eq3.42}.
\end{proof}
\end{theorem}

We also see
\begin{lemma}\label{p4_convergence}
Assume $\calJ$ is a Jacobi matrix asymptotic to type $1$, and let its Jacobi parameters satisfy $(A1)$, $(A2)$, or $(A3)$. Then uniformly on the compacts of the appropriate region,
\begin{align*}
u(z;\calJ^{(n)}) &\to \bdone, \\
M(z;\calJ^{(n)}) & \to\bdone z,
\end{align*}
where $u^{(n)}$ and $M^{(n)}$ are the Jost function and the $M$-function, respectively, for the $n$ times stripped operator $\calJ^{(n)}$.
\begin{proof}
%
Note that $M^{(n)}(z)=z u(z;\calJ^{(n)})^{-1} u(z;\calJ^{(n+1)})A_{n+1}^{-1}=A_n^{-1} u_n(z;\calJ)^{-1} u_{n+1}(z;\calJ)$
. But $A_n\to\bdone$ and $z^{-n} u_n(z)\to\bdone$ uniformly on compacts of the appropriate region by \eqref{p3_eq21}. This and \eqref{p3_eq19} give the result.
\end{proof}
\end{lemma}

To end this section, we get the following result for free as a corollary from Theorems \ref{p3_th5}, \ref{p3_th2}, and \cite{K_sz}. The scalar analogue is proven in Killip--Simon \cite[Thm 9.14]{KS}.

\begin{theorem}\label{p3_th7000}
Let $\calJ$ be of type asymptotic to type $1$ and satisfies $(A1)$
. Then $u(z;\calJ)$ has the following factorization:
\begin{equation*}
u(z;\calJ)=U B(z) O(z),
\end{equation*}
where $U$ is a constant unitary  matrix, $B(z)$ is a matrix-valued Blaschke-Potapov product with zeros at $\{z_j\}$, and $O(z)$ is a matrix-valued outer function, uniquely defined from the conditions
\begin{equation}\label{p3_eq3.56}
\begin{gathered}
O(e^{i\theta})^* O(e^{i\theta}) = \sin\theta \,\left(\imm M(e^{i\theta})\right)^{-1}, \\
O(0) =O(0)^* > \bdnot ,\\
\log\left|\det O(e^{i\theta})\right| = \int_{-\pi}^\pi \log \left| \det O(e^{i\theta}) \right| \frac{d\theta}{2\pi}.
\end{gathered}
\end{equation}
In particular, $u$ has trivial singular inner part.
\begin{remarks}
1. That the outer factor $O$ can be uniquely defined from the conditions \eqref{p3_eq3.56}, as long as \eqref{p3_eq3.57} holds, 
is Lemma \ref{p2_lm2}.

2. $O$  has an integral representation in terms of Potapov multiplicative integral, see \cite{K_sz} for the details.
\end{remarks}
\begin{proof}
By Theorem~\ref{p3_th2} $u(z;\calJ)=(1-z^2) L(z)$, where $L(z)=\lim_{n\to\infty} z^n \mathfrak{p}_n(z+z^{-1})$. By the results from \cite{K_sz}, $L(z)$ is an $H^2(\bbD)$ function with no singular inner part. Since $1-z^2$ is a bounded outer function, $u$ is an $H^2(\bbD)$ function with no singular inner part as well.

By \eqref{p3_eq24}, $u(e^{i\theta};\calJ)^* u(e^{i\theta};\calJ)=\sin\theta \,\left(\imm M(e^{i\theta})\right)^{-1}$, and so \eqref{p3_eq3.56} has to hold. Note that
\begin{equation}\label{p3_eq3.57}
\int_{-\pi}^\pi \log\det \left[\sin\theta (\imm M(e^{i\theta}))^{-1}\right] \dt> -\infty
\end{equation}
 is equivalent to
\begin{equation*}\label{p3_eq3.58}
\left| \int_{-2}^2 (4-x^2)^{-1/2} \log\det f(x) dx \right|  <\infty,
\end{equation*}
which is indeed finite given  $(A1)$ (see \cite[Section 14]{DKS}).
\end{proof}
\end{theorem}


\section{The Inverse Problem}

Now we start with an analytic function $u$ and seek to construct such a measure $\mu$ that $u$ is its Jost function. We do this in Subsection \ref{p3_sub31}. In the proof of Theorem \ref{p3_th_constr} however, we appeal to the results later in the section. Note that this theorem is never used in Subsections \ref{p3_sub32} and \ref{p3_sub33} (i.e., we are never assuming that $u$ is actually the Jost function for $\mu$). In Subsections \ref{p3_sub32} and \ref{p3_sub33} we derive the exponential decay of the Jacobi parameters of $\mu$, proving Theorems \ref{p3_th10} and \ref{p3_th11}. Subsection \ref{p3_sub34} is just a restatement of the results in terms of the so-called perturbation determinants.

Throughout this section let $u$ be an analytic function in $\bbD_R$ for some $R>1$ satisfying the conditions of Theorem \ref{p3_th_constr}. Note that by \eqref{p2_herg1}--\eqref{p2_herg2} and \eqref{p3_eq24} the absolutely continuous part $f(x)$ of $\mu$ is forced to be $f(2\cos\theta)=\pi^{-1}\left|\sin\theta\right|\,\left[u(e^{i\theta})^* u(e^{i\theta})\right]^{-1}$, and its singular part to be pure point with some weights $w_j$ at $E_j=z_j+z_j^{-1}$, where $z_j$ are zeros of $u$ in $\bbD$. By Theorem \ref{p3_th5}(iv), $w_j$ must satisfy the condition (ii) of Theorem \ref{p3_th_constr}. Assuming also (i), this $\mu$ is a probability measure. Its $M$-function satisfies \eqref{p2_herg1}, so
\begin{equation}\label{p3_eq34}
\imm M(e^{i\theta})=\sin\theta\,\left[u(e^{i\theta})^* u(e^{i\theta})\right]^{-1}
\end{equation}
holds. Just as in the proof of Theorem~{\ref{p3_th5}}(viii), assuming ~\eqref{eq:symmetry} we can extend $M$ meromorphically to $\bbD_R$ and see that
\begin{equation*}
M(z)=M^\sharp(z)+(z-z^{-1})\left[u^\sharp(z)u(z)\right]^{-1}, \quad R^{-1}<|z|<R.
\end{equation*}

Let $\calJ$ with Jacobi parameters $\{A_n,B_n\}_{n=1}^\infty$ be the type $1$ Jacobi matrix for $d\mu$.

Define inductively
\begin{align}
u^{(n+1)}(z)&=z^{-1}u^{(n)}(z)M^{(n)}(z)A_{n+1};\label{p3_eq41}\\
A_{n+1} M^{(n+1)}(z)A_{n+1}^*&=\left(z+\frac1z\right)\mathbf{1}-B_{n+1}-{M^{(n)}(z)}^{-1}.\label{p3_eq42}
\end{align}
Then $M^{(n)}$ is the $M$-function for $\calJ^{(n)}$ and, by an easy induction,
\begin{equation}\label{p3_eq4.5}
M^{(n)}(z)={M^{(n)}}{}^\sharp(z)+(z-z^{-1})\left[{u^{(n)}}{}^\sharp(z) u^{(n)}(z)\right]^{-1}, \quad R^{-1}<|z|<R,
\end{equation}
 holds.

\medskip
\subsection{Proof of Theorem \ref{p3_th_constr}
}\label{p3_sub31}

\begin{remark} It is clear that any two matrices having $u$ as its Jost function are asymptotic to each other, and moreover,
related by $\widetilde{\calJ}=U\calJ U^{-1}$, where $U$ is an $l\times l$ block diagonal unitary $U=\sigma_1\oplus\sigma_2\oplus\sigma_3\oplus\ldots$, where $\sigma_n$ are unitary with $\sigma_1=\mathbf{1}$ and $\lim_{n\to\infty}\sigma_n=\mathbf{1}$ (which is a stronger condition than just being asymptotic).
\end{remark}

\begin{proof}
The results of this section show that $||B_n||$ and $||\mathbf{1}-A_nA_n^*||$ decay exponentially (with the rate $r^{-2n}$, where $r$ could be only slightly larger than $1$). Thus the Jost function $\widetilde{u}$ exists and is analytic in $\bbD_r$. Consider
\begin{equation*}\label{p3_eq4.2}
g(z)=\widetilde{u}(z)u(z)^{-1}.
\end{equation*}
We want to prove $g$ is analytic and nonvanishing. Since $u^{-1}$ has a first order pole at $z_j$, $\widetilde{u} u^{-1}$ is analytic at $z_j$ if and only if
\begin{equation}\label{p3_eq4.3}
\widetilde{u}(z_j) \res_{z=z_j} u(z)^{-1}=\bdnot,
\end{equation}
which is equivalent to the condition $\ran \res_{z=z_j} u(z)^{-1}\subseteq \ker \widetilde{u}(z_j)$. However by Lemma \ref{p4_lm0}, $$\ran \res_{z=z_j} u(z)^{-1}=\ker u(z_j),$$ which equals to $\ran w_j$ by the condition (ii). By Theorem \ref{p3_th5}(iv), $\ran w_j=\ker \widetilde{u}(z_j)$, and \eqref{p3_eq4.3} follows.


$g(z)$ is analytic at $\pm1$ by the following arguments. By~(\ref{p3_eq24}) and~(\ref{p3_eq34}),
\begin{equation*}
u(\pm1)^* u(\pm1)=\widetilde{u}(\pm1)^* \widetilde{u}(\pm1).
\end{equation*}
This implies $\ker u(\pm1)=\ker \widetilde{u}(\pm1)$ (since $\ker T=\ker T^* T$), and then identical arguments as for $z_j$'s show that  $g(z)$ is analytic at $\pm1$.

Thus we have proved $g$ is analytic on a neighborhood of $\overline{\D}$, and switching the roles of $u$ and $\widetilde{u}$, we obtain that $g$ is also non-vanishing there.

Now,
\begin{equation*}
\begin{aligned}
g(z)^* g(z)=[u(z)^{-1}]^* \widetilde{u}(z)^* \widetilde{u}(z) u(z)^{-1}&=\sin\theta \, [u(z)^{-1}]^* [\imm M(e^{i\theta})]^{-1} u(z)^{-1} \\&=[u(z)^{-1}]^* u(z)^* u(z) u(z)^{-1}=\mathbf{1}.
\end{aligned}
\end{equation*}
So $g(z)$ is analytic and invertible on $\overline{\D}$ and unitary on $\partial \D$, which implies (e.g., by the Schwarz reflection) that $g(z)\equiv v_0$ for some constant unitary $v_0$. Thus, $u(z)=v_0^*\widetilde{u}(z)$. Theorem~\ref{p3_th3} implies that $u$ is the Jost function for the Jacobi matrix with parameters $(A_1 v_0,v_0^* A_2 v_0,v_0^* A_3 v_0,\ldots)$, $(B_1,v_0^* B_2 v_0,v_0^* B_3 v_0,\ldots)$.
\end{proof}

\subsection{Proof of Theorems \ref{p3_th10} and \ref{p3_th11} for the case of no bound states}\label{p3_sub32}

In this subsection we prove Theorems \ref{p3_th10} and \ref{p3_th11} for the case when $\mu$ has no bound states. Thus these theorems take the following form.

\begin{theorem}\label{p3_th7}
Let $u(z)$ be a polynomial obeying~\eqref{eq:symmetry} and
\begin{itemize}
\item[(i)] $u(z)$ is invertible on $\overline{\D}\setminus \{\pm1\}$;
\item[(ii)] if $\pm1$ are zeros, they are simple;
\item[(iii)]$\frac2\pi \int_{0}^\pi \sin^2\theta\,\left[u(e^{i\theta})^* u(e^{i\theta})\right]^{-1}d\theta=\mathbf{1}$.
\end{itemize}
Then $u$ is the Jost function of a Jabobi matrix with
\begin{equation}\label{p3_eq4.9}
\mathbf{1}-A_nA_n^*=B_n=\mathbf{0} \quad \mbox{ for all large } n.
\end{equation}
\end{theorem}

\begin{theorem}\label{p3_th8}
Let $u(z)$ be analytic in $\bbD_R$ for some $R>1$ and obeys ~\eqref{eq:symmetry} and (i)--(iii) from Theorem~{\ref{p3_th7}}, then $u$ is the Jost function of a Jacobi matrix with
\begin{equation}\label{p3_eq4.10}
\limsup_{n\to\infty}\left(||B_n||+||\mathbf{1}-A_nA_n^*||\right)^{1/2n}\le R^{-1}.
\end{equation}
\end{theorem}

\begin{remark}
We denoted $\{A_n,B_n\}_{n=1}^\infty$ to be the type $1$ Jacobi coefficients for $d\mu$. $u$ will be the Jost function for a different Jacobi matrix (asymptotic to it). However \eqref{p3_eq4.9} and \eqref{p3_eq4.10} are invariant within the class of equivalent Jacobi matrices.
\end{remark}
%

%
%

Note that \eqref{p3_eq41} and \eqref{p3_eq42} define $u^{(n)}$ and $M^{(n)}$, which are in general meromorphic functions in $\bbD_R$. We will show below that $u^{(n)}$ are actually analytic. Let us first prove the following lemma.

\begin{lemma}\label{p3_jost_boundary}
Let $u^{(n)}$ and $M^{(n)}$ be given by \eqref{p3_eq41} and \eqref{p3_eq42}. Then $u^{(n)}$ has no zeros on $\partial\bbD$ except possibly at $\{\pm1\}$, in which case they are simple.
\begin{proof}
Since \eqref{p3_eq4.5} holds, we obtain
\begin{equation*}
f^{(n)}(2\cos\theta)=\pi^{-1} \left| \imm M^{(n)}(e^{i\theta}) \right| =\pi^{-1} | \sin\theta\ | \left[u^{(n)}(e^{i\theta})^* u^{(n)}(e^{i\theta})\right]^{-1},
\end{equation*}
where $f^{(n)}$ is the density of the spectral measure $\mu^{(n)}$ of $\calJ^{(n)}$. Since $\int_{-\pi}^\pi |\sin\theta| f^{(n)}(2\cos\theta)d\theta \le\mu^{(n)}(\bbR)\le\bdone$, we get the result.
\end{proof}
\end{lemma}

Now we can obtain analyticity of $u^{(n)}$ for $n\ge1$.

\begin{theorem}\label{p3_th9}
If $u$ is analytic in $\bbD_R$ and nonvanishing on $\overline{\D}\setminus\{\pm1\}$ with at most simple zeros at $\pm1$, then the same is true of each $u^{(n)}$.
\end{theorem}
\begin{proof}
We use induction on $n$. The inductive hypothesis will be to assume
\begin{itemize}
\item[(a)] $u^{(n)}$ is analytic in $\bbD_R$,
\item[(b)] $u^{(n)}$ is invertible on $\overline{\D}\setminus\{\pm1\}$,
\item[(c)] $u^{(n)}$ has at most simple zeros at $\pm1$,
\item[(d)] $M^{(n)}$ has no poles in $\overline{\bbD}\setminus\{\pm1\}$,
\item[(e)] $M^{(n)}$ has at most simple poles at $\pm1$,
\item[(f)] $(M^{(n)})^{-1}$ has no poles in $\overline{\bbD}\setminus\{\pm1\}$,
\item[(g)] $(M^{(n-1)})^{-1}$ has at most simple poles at $\pm1$.
\end{itemize}
Let us check the base case $n=0$. (a)--(c) are given. That $M$ has no poles in $\bbD$ follows from the fact that $\mu$ has no eigenvalues outside $[-2,2]$, and no poles of $M$ on $\partial\bbD\setminus\{\pm1\}$ corresponds to the absence of the point spectrum in $(-2,2)$. Also, no point spectrum at $\pm2$ implies  $\lim_{\veps\downarrow 0} \veps \mathfrak{m}(\pm2+i\veps)=0$ which translates to $\lim_{z\to \pm1} (z\mp 1)^2 M(z)=0$. Thus we established (d) and (e).

Observe that $M$ cannot have zeros on $(-1,0)\cup(0,1)$ since this would correspond to $\int_{-2}^2 \frac{d\mu(x)}{x-z}$ being singular at some real $z$ with $|z|>2$. On $\{z\in\D\mid\imm z>0\}$ we have $\imm M(z)>\bdnot$, so $M$ is invertible. Same for $\{z\in\D\mid\imm z<0\}$. Finally, $M$ is also invertible on $\partial\D\setminus\{\pm1\}$ since $\imag M$ is invertible there by~(\ref{p3_eq34}). Thus $M^{-1}$ has no poles in $\overline{\bbD}\setminus\{\pm1\}$, i.e., (f) holds.

(g) is vacuous for $n=0$.



Now assume that (a)--(g) hold for $n$, and let us show they hold for $n+1$ as well. By (d) $M^{(n)}$ is meromorphic on $\bbD_R$ with poles possible only in $\{z\mid 1<|z|<R\}\cup\{\pm1\}$. Using
\begin{equation}\label{p3_eq43}
M^{(n)}(z)=M^{(n)}{}^\sharp(z)+(z-z^{-1})\left[u^{(n)}{}^\sharp(z)u^{(n)}(z)\right]^{-1}, \quad R^{-1}<|z|<R,
\end{equation}
we see the following:

\begin{itemize}
\item[(i)] $M^{(n)}$ has a pole at $z_k$, $1<|z_k|<R$, only if $u^{(n)}(z_k)$ is not invertible, since $u^{(n)}{}^\sharp(z_k)$ is invertible by (b) and $M^{(n)}{}^\sharp(z_k)$ is regular by (d). Then~(\ref{p3_eq41}) and~(\ref{p3_eq43}) imply
    \begin{equation*}
    u^{(n+1)}(z_k)=z_k^{-1}u^{(n)}(z_k)M^{(n)}{}^\sharp(z_k)A_{n+1}+(1-z_k^{-2})[u^{(n)}{}^\sharp(z_k)]^{-1}A_{n+1}
    \end{equation*}
    is regular.
\item[(ii)] Assume $M^{(n)}$ has a pole at $\pm 1$. By (c) and (e), $u^{(n)}$ and $M^{(n)}$ have at most order $1$ poles at $\pm 1$, so let
    \begin{align*}
    \res_{z=1}M^{(n)}(z)=T 
    \\
    \res_{z=1}u^{(n)}(z)^{-1}=C.
    \end{align*}
    From the definition of $M^{(n)}$, the matrix $T$ must be Hermitian. Easy to see,
    \begin{align*}
    \res_{z=1}M^{(n)}{}^\sharp(z)=-T^*=-T,\\
    \res_{z=1}u^{(n)}{}^\sharp(z)^{-1}=-C^*,
    \end{align*}
    and then computing residues of both sides of~(\ref{p3_eq43}) gives
    \begin{equation*}\label{p3_eq49}
    2T=-2CC^*.
    \end{equation*}
    Now, by~(\ref{p3_eq41}),
    \begin{equation*}
    \res_{z=1}u^{(n+1)}(z)= \lim_{z\to1}(z-1)u^{(n+1)}(z)=u^{(n)}(1)T A_{n+1}=- u^{(n)}(1)CC^* A_{n+1}=\mathbf{0},
    \end{equation*}
    since $\ran C=\ker u^{(n)}(1) $ (by Lemma \ref{p4_lm0}). Hence $u^{(n+1)}$ is regular at $z=\pm1$.
\end{itemize}

This proves part (a) of the inductive step.

$u^{(n+1)}$ is invertible on $\overline{\D}\setminus\{\pm1\}$ since $u^{(n)}$ is invertible and $(M^{(n)})^{-1}$ has no poles (by (b) and (f)). This establishes (b).

(c) is obtained in Lemma \ref{p3_jost_boundary}.

(d) for $n+1$ follows from \eqref{p3_eq42} and (f) for $n$.

(f) for $n+1$ follows by the exact same arguments as for $n=0$ before.

(g) follows from $M^{(n)}(z)^{-1} = z^{-1} A_{n+1} u^{(n+1)}(z)^{-1} u^{(n)}(z)$ and Lemma \ref{p3_jost_boundary}.

Finally, (e) follows from \eqref{p3_eq42} since we just established that $M^{(n)}(z)^{-1}$ has at most simple poles at $\pm1$.
\end{proof}

Note that $\esssup \mu=[-2,2]$ with $\det f(x)>0$ on $(-2,2)$, and so Denisov--Rakhmanov theorem (Lemma \ref{denisov}) implies that $\calJ$ is in the Nevai class. By Theorem \ref{p1_thm1} we obtain $A_n\to\bdone, B_n\to\bdnot$. This means that $\calJ^{(n)}$ converges in norm to the free block Jacobi matrix, which implies that resolvents converge:
\begin{equation}\label{p3_eq4.23} M^{(n)}(z)\to z\bdone \quad \mbox{  uniformly on compacts of }\bbD.\end{equation}

Now combine \eqref{p3_eq41} and \eqref{p3_eq4.5} to get
\begin{equation}\label{p3_eq4.24}
u^{(n+1)}(z)= (1-z^{-2}) ({u^{(n)}} {}^\sharp(z))^{-1} A_{n+1}+z^{-2} u^{(n)}(z) N_n^\sharp(z) A_{n+1},
\end{equation}
where $N_n(z)=M^{(n)}(z)/z$, $N_n^\sharp(z)=z M^\sharp(z)$.

Let us fix any $R_1$ with $1<R_1<R$. Given any $L^2(\bdone \frac{d\theta}{2\pi})$ function on $R_1\partial\bbD$, define
\begin{equation*}
|||f|||_{R_1}=\left( \int_{-\pi}^\pi \left|\left| (P_+f)(R_1 e^{i\theta})\right|\right|^2 \dt \right)^{1/2},
\end{equation*}
where $P_+$ is the projection in $L^2(\bdone \frac{d\theta}{2\pi})$ onto $\{e^{in\theta}\}_{n=1}^\infty$, and $||\cdot||$ is the Hilbert-Schmidt norm till the end of this section. In particular, if $f$ is analytic in $\bbD_R$,
$$
|||f|||_{R_1}=\left( \int_{-\pi}^\pi \left|\left| f(R_1 e^{i\theta})-f(0)\right|\right|^2 \dt \right)^{1/2}.
$$

Now note that since $(u^{(n)} {}^\sharp)^{-1}$ is analytic in $(\bbC\cup\{\infty\})\setminus \bbD$, $P_+((1-z^{-2}) ({u^{(n)}} {}^\sharp(z))^{-1} A_{n+1})=\bdnot$. For the same reasons, $P_+(z^{-2} u^{(n)}(0) N_n^\sharp(z) A_{n+1})=\bdnot$. Thus
\begin{equation*}
P_+(u^{(n+1)})=P_+\left(z^{-2} (u^{(n)}(z)-u^{(n)}(0)) N_n^\sharp(z) \right)A_{n+1}.
\end{equation*}

Since $P_+$ is a projection on $L^2$, using submultiplicativity of the Hilbert-Schmidt norm we get
\begin{equation*}
|||u^{(n+1)}|||_{R_1}\le R_1^{-2} |||u^{(n)}|||_{R_1}  \,||A_{n+1}||  \sup_{|z|=R_1} ||N_n^\sharp(z)||,
\end{equation*}
which by induction gives
\begin{equation}\label{p3_eq4.28}
|||u^{(n+1)}|||_{R_1}\le R_1^{-2n} |||u|||_{R_1}  \,\left[ \prod_{j=1}^{n}||A_{j+1}||  \sup_{|z|=R_1} ||N_j^\sharp(z)|| \right].
\end{equation}

Now since $||A_j||\to\bdone$ and $\sup_{|z|=R_1} ||N_j^\sharp(z)||\le \sup_{|z|\le R^{-1}_1} ||M^{(j)}(z)/z||\to \bdone$ by \eqref{p3_eq4.23}, we get that for any $\veps>0$ there exists a constant $c_\veps$ such that
$$\left[ \prod_{j=1}^{n}||A_{j+1}||  \sup_{|z|=R^{-1}_1} ||N_j(z)|| \right]\le c_\veps (1+\veps)^{2n},$$
and so
\begin{equation}\label{p3_eq4.29}
|||u^{(n+1)}|||_{R_1}\le C_\veps (R_1-\veps)^{-2n}
\end{equation}
for some new constant $C_\veps$.

\begin{proof}[Proof of Theorem~\ref{p3_th7}]
Since $u$ is a polynomial, then taking $n$ and $R_1$ sufficiently large in \eqref{p3_eq4.28}, one can see that $|||u^{(n)}|||_{R_1}=0$, which implies $u^{(n)}(z)=u^{(n)}(0)$. Then by the condition (iii) of the theorem, $u^{(n)}(z)=1$, and so $f^{(n)}(2\cos\theta)=\pi^{-1}|\sin\theta|$ is free, that is, $\mathbf{1}-A_nA_n^*=B_n=\mathbf{0}$ for all large $n$.
\end{proof}

\begin{proof}[Proof of Theorem~\ref{p3_th8}]
Define $s_n(z)=u^{(n)}(z)u^{(n)}(0)^{-1}-\bdone$. Note that by Szeg\H{o} asymptotics (Theorem \cite[Thm 1]{K_sz}), the limit $z^n \mathfrak{p}_n(z+z^{-1})$ exists. In particular at $z=0$ this gives that there exists $\lim_{n\to\infty} A_1\ldots A_n\equiv K$, with $K$ invertible. Then $u^{(n)}(0)=u(0)A_1\ldots A_n\to u(0)K$ is bounded in norm from above and below away from $0$. Then
$$
|||s_n|||_{R_1}\le |||u^{(n)}|||_{R_1} \, ||u^{(n)}(0)^{-1}|| \le C_\veps (R_1-\veps)^{-2n}
$$
for some new constant $C_\veps$. Using Cauchy formula, one easily obtains from this
\begin{equation}\label{p3_eq4.30}
||s_n(z)||\le \widetilde{C}_\veps (R_1-\veps)^{-2n} \quad \mbox{uniformly in } \bbD_{R_1-2\veps}.
\end{equation}

Now note that by \eqref{p3_eq41} $$\frac{M^{(n)}(z)}z=u^{(n)}(z)^{-1} u^{(n+1)}(z) A_{n+1}^{-1}=u^{(n)}(0)^{-1}(\bdone+s_n(z))^{-1} (\bdone+s_{n+1}(z)) u^{(n)}(0) ,$$
and so
\begin{multline*}
\sup_{|z|\le1/2} \left|\left| \frac{M^{(n)}(z)}{z}-\bdone\right|\right| \le
\sup_{|z|\le1/2} \left|\left| u^{(n)}(0)^{-1}(\bdone+s_n(z))^{-1}  u^{(n)}(0) -\bdone\right|\right|
\\ +\sup_{|z|\le1/2} \left|\left| u^{(n)}(0)^{-1}(\bdone+s_n(z))^{-1} s_{n+1}(z) u^{(n)}(0) \right|\right|.
\end{multline*}
The second term can be made exponentially small simply by using \eqref{p3_eq4.30}, while the first is
\begin{equation*}
\begin{aligned}
\left|\left| u^{(n)}(0)^{-1}(\bdone+s_n(z))^{-1}  u^{(n)}(0) -\bdone\right|\right| &=
\left|\left| u^{(n)}(0)^{-1}\sum_{j=0}^\infty s_n(z)^j  u^{(n)}(0) -\bdone\right|\right| \\
&= \left|\left| u^{(n)}(0)^{-1}\sum_{j=1}^\infty s_n(z)^j  u^{(n)}(0) \right|\right| \\&\le
||u^{(n)}(0)^{-1}||\, ||u^{(n)}(0)|| \frac{||s_n(z)||}{1-||s_n(z)||}
\end{aligned}
\end{equation*}
which is also uniformly exponentially small. Thus
\begin{equation*}
\sup_{|z|\le1/2} \left|\left| \frac{M^{(n)}(z)}{z}-\bdone\right|\right| \le  \widehat{C}_\veps (R_1-\veps)^{-2n}.
\end{equation*}
Using this, \eqref{p3_m-taylor}, and the Cauchy formula, we obtain
\begin{equation*}
||B_n||+||\mathbf{1}-A_nA_n^*|| \le \widehat{C}_\veps (R_1-\veps)^{-2n}.
\end{equation*}
Since $R_1<R$ and $\veps>0$ were arbitrary, we obtain \eqref{p3_eq4.10}.
\end{proof}

Note that instead of $1/2$ we could have taken any constant smaller than $R_1-\veps$ here. Therefore we have shown that $M^{(n)}(z)\to z\bdone$ uniformly on compacts of $\bbD_R$.

\subsection{Proof of Theorems \ref{p3_th10} and \ref{p3_th11} for the general case}\label{p3_sub33}
Recall Definition \ref{canonic} of canonical weight: $w_j$ is canonical if
\begin{equation}\label{one}
\widetilde{w}_j \, u(1/{\bar{z}_j})^*=-(z_j-z_j^{-1})\lim_{z\to z_j}(z-z_j)u(z)^{-1},
\end{equation}
where as before $w_j=(z_j^{-1}-z_j)z_j^{-1} \widetilde{w}_j$.
As clear from the calculation in Theorem \ref{p3_th5}(ix), the weight is canonical if and only if $u^{(1)}(z)$ is regular at $z_j^{-1}$.

\begin{lemma}\label{p3_lm9}
Assume $u(z)$ and $u^{(1)}(z)$ are analytic in $\bbD_R$. Then for any $n\ge2$, $u^{(n)}(z)$ is analytic in $\bbD_R$.
\begin{proof}
Note that part (vii) of Theorem \ref{p3_th5} can be proved using only \eqref{p3_eq22} and \eqref{p3_m-recur}. Therefore \eqref{p3_eq41} and \eqref{p3_eq42} allow us to conclude that
\begin{equation*}
u^{(n+2)}(z)
=z^{-1} u^{(n+1)}(z) A_{n+1}^{-1}\left((z+z^{-1})\bdone-B_{n+1}\right) A_{n+1}^*{}^{-1} A_{n+2}-
z^{-2} u^{(n)}(z)  {A_{n+1}^*}^{-1} A_{n+2},
\end{equation*}
which proves our statement (easy to see that $z=0$ in fact is not causing any troubles here).
\end{proof}
\end{lemma}
\begin{remark}
What this lemma says is that if all the weights of $u$ are canonical, then they are automatically canonical for every $u^{(n)}$.
\end{remark}

For the inductive step in this case we will need the following result.

\begin{lemma}\label{p3_lm10}
If $u$ and $M$ satisfy
\begin{itemize}
\item[(a)]
$\ker u(\xi)=\ran \res_{z=\xi} M(z)$ for all $\xi\in\bbD$;
\item[(b)] all poles of $u^{-1}$ in $\overline{\bbD}\cap\bbR$ are simple,
\end{itemize}
 then the same is true for all $u^{(n)}$ and $M^{(n)}$.
\begin{proof}
Assume both conditions hold for $u^{(n)}$ and $M^{(n)}$.

Take any $\xi\in\bbD$. Note that in the Smith--McMillan form (Lemma \ref{p4_lmSM}) of $u^{(n)}$ at $z=\xi$ each power $\kappa_j$ of $(z-\xi)^{\kappa_j}$ must be $0$ or $1$ by (b). Thus
\begin{equation*}
u^{(n)}(z)=E(z) \left( \begin{array}  {cc} (z-\xi) \bdone_s & \bdnot \\ \bdnot & \bdone_{l-s}
\end{array}\right)
 F(z),
\end{equation*}
where $\bdone_j$ is the $j\times j$ identity matrix. Now since $M^{(n+1)}$ can have only first order poles in $\bbD$, it means that $M^{(n)}$ can have only first order zeros/poles in $\bbD$. Then the Smith--McMillan form of  $(M^{(n)})^{-1}$ at $\xi$ is
\begin{equation*}
M^{(n)}(z)^{-1}=G(z) \left( \begin{array}  {ccc} (z-\xi) \bdone_p & \bdnot &\bdnot \\ \bdnot & \bdone_{q} &\bdnot \\ \bdnot &\bdnot &\frac{1}{z-\xi}\bdone_{l-p-q}
\end{array}\right) H(z).
\end{equation*}

Observe that $E(z), F(z), G(z), H(z)$ are analytic and invertible in a neighborhood of $\xi$.

Now note that $$\ker u^{(n)}(\xi)=F(\xi)^{-1}\spann\{\delta_{1},\ldots,\delta_s\},$$ and $$\ran\res_{z=\xi} M^{(n)}(z)=H(\xi)^{-1}\spann\{\delta_1,\ldots,\delta_p\}.$$ Then the condition (a) implies that $s=p$, and that $\spann\{\delta_{1},\ldots,\delta_p\}$ is an invariant subspace of the matrix $V\equiv H(\xi)F(\xi)^{-1}$. Thus
\begin{equation*}
V=\left( \begin{array}  {cc} V_{11} & V_{12} \\ \bdnot & V_{22}
\end{array}\right),
\end{equation*}
where $V_{11}$ is an (invertible) $p\times p$ matrix, $V_{22}$ is an (invertible) $(l-p) \times (l-p)$ matrix, and $V_{12}$ is an $s\times(l-p)$ matrix.

By (a) $u^{(n+1)}(z)$ is analytic at $\xi$. Now consider $u^{(n+1)}(z)^{-1}$ at $z=\xi$. We want to show the following limit is finite:
\begin{multline}\label{p3_eq4.40}
\lim_{z\to\xi} (z-\xi) u^{(n+1)}(z)^{-1}=A_{n+1}^{-1} \lim_{z\to\xi} (z-\xi)  M^{(n)}(z)^{-1} u^{(n)}(z)^{-1} \\
=A_{n+1}^{-1} G(\xi) \lim_{z\to\xi} (z-\xi) \left( \begin{array}  {ccc} (z-\xi) \bdone_p & \bdnot &\bdnot \\ \bdnot & \bdone_{q} &\bdnot \\ \bdnot &\bdnot &\frac{1}{z-\xi}\bdone_{l-p-q}
\end{array}\right) V
\left( \begin{array}  {ccc} \frac{1}{z-\xi} \bdone_p & \bdnot &\bdnot \\ \bdnot & \bdone_{q} &\bdnot \\ \bdnot &\bdnot &\bdone_{l-p-q}
\end{array}\right)
 E(\xi)^{-1}.
\end{multline}
But
$$
\left( \begin{array}  {cc} (z-\xi)\bdone_p & \bdnot \\ \bdnot & \bdone_{l-p}
\end{array}\right)
\left( \begin{array}  {cc} V_{11} & V_{12} \\ \bdnot & V_{22}
\end{array}\right)
\left( \begin{array}  {cc} \frac{1}{z-\xi}\bdone_p & \bdnot \\ \bdnot & \bdone_{l-p}
\end{array}\right)
=
\left( \begin{array}  {cc} V_{11} & (z-\xi) V_{12} \\ \bdnot & V_{22}
\end{array}\right),
$$
which means that the right-hand side of \eqref{p3_eq4.40} is equal to
\begin{multline}\label{p3_eq4.41}
A_{n+1}^{-1} G(\xi) \lim_{z\to\xi} (z-\xi) \left( \begin{array}  {ccc} \bdone_p & \bdnot &\bdnot \\ \bdnot & \bdone_{q} &\bdnot \\ \bdnot &\bdnot &\frac{1}{z-\xi}\bdone_{l-p-q}
\end{array}\right) \widetilde{V}
\left( \begin{array}  {ccc} \bdone_p & \bdnot &\bdnot \\ \bdnot & \bdone_{q} &\bdnot \\ \bdnot &\bdnot &\bdone_{l-p-q}
\end{array}\right)
 E(\xi)^{-1} \\
 = A_{n+1}^{-1} G(\xi) \left( \begin{array}  {ccc} \bdnot_p & \bdnot &\bdnot \\ \bdnot & \bdnot_{q} &\bdnot \\ \bdnot &\bdnot &\bdone_{l-p-q}
\end{array}\right) \widetilde{V}
 E(\xi)^{-1} ,
\end{multline}
where $\widetilde{V}=\left( \begin{array}  {cc} V_{11} & \bdnot \\ \bdnot & V_{22}
\end{array}\right).$ This establishes (b) for $u^{(n+1)}$ for $\xi\in\bbD\cap \bbR$. The fact that $\pm1$ is at most first order pole of $(u^{(n+1)})^{-1}$ is already proved in Lemma \ref{p3_jost_boundary}.

To show that (a) holds for $u^{(n+1)}$, note that by Lemma \ref{p4_lm0} (which applies since we already know that $(u^{(n+1)})^{-1}$ has at most simple pole),
$$\ker u^{(n+1)}(\xi)=\ker u^{(n)}(\xi) M^{(n)}(\xi) A_{n+1}= \ran \res_{z=\xi} A_{n+1}^{-1} \left(M^{(n)}(z)^{-1} u^{(n)}(z)^{-1}\right),$$
and by \eqref{p3_eq42},
$$ \ran\res_{z=\xi} M^{(n+1)}(z)= \ran \res_{z=\xi} A_{n+1}^{-1} M^{(n)}(z)^{-1}.$$

By the calculations \eqref{p3_eq4.40}--\eqref{p3_eq4.41} above, it is easy to see that both of these spaces are equal to
\begin{equation*}
\ran A_{n+1}^{-1} G(\xi) \left( \begin{array}  {ccc} \bdnot_p & \bdnot &\bdnot \\ \bdnot & \bdnot_{q} &\bdnot \\ \bdnot &\bdnot &\bdone_{l-p-q}
\end{array}\right).
\end{equation*}
\end{proof}
\end{lemma}

%

This gives us the analogue of Theorem \ref{p3_th9}:

\begin{lemma}\label{inductive}
If $u$ is analytic in $\bbD_R$, satisfies (a)--(b) of Lemma \ref{p3_lm10}, and all the weights with $1>|z_j|>R^{-1}$ are canonical, then the same is true of each $u^{(n)}$.
\begin{proof}
The arguments of Theorem \ref{p3_th9}, together with the result of Lemma \ref{p3_lm10}, give the result. Note that condition (a) ensures analyticity of $u^{(1)}$ at $z_j$, and canonic weights ensure analyticity of $u^{(1)}$ at $z_j^{-1}$. The weights for $u^{(n)}$ for $n\ge1$ are canonical by Lemma \ref{p3_lm9}.
\end{proof}
\end{lemma}

\begin{proof}[Proof of Theorem~\ref{p3_th10}]
If some of the weights are not canonical then $u^{(1)}$ is not entire, and so $\mathbf{1}-A_nA_n^*=B_n=\mathbf{0}$ cannot hold for all large $n$.

Now assume all the weights are canonical. 
Then all $u^{(n)}$'s are entire by Lemma \ref{inductive}. For $r$ sufficiently large, \eqref{p3_eq4.24} implies
\begin{equation*}
\sup_{|z|\le r} ||u^{(n+1)}(z)|| \le O(1) \left( 1+r^{-2} \sup_{|z|\le r} ||u^{(n)}(z)||\right),
\end{equation*}
which inductively shows that if $u$ is a polynomial then $u^{(n)}$ is a polynomial with
$$\deg u^{(n)}\le \max\{0,\deg u-2n\}.$$
Then $u^{(n)}$ is a constant for some large $n$. By Lemma \ref{p3_lm10}, $M^{(n)}$ has no poles, and so \eqref{p3_eq4.5} implies that $u^{(n)}$ satisfies the condition (iii) of Theorem \ref{p3_th7} (as well as conditions (i) and (ii), of course). This implies $\mathbf{1}-A_nA_n^*=B_n=\mathbf{0}$ for all large $n$.
\end{proof}

\begin{proof}[Proof of Theorem~\ref{p3_th11}]
If some of the weights with $1>|z_j|>R^{-1}$ are not canonical then $u^{(1)}$ is not analytic at $\{z_j^{-1}\}$, and so $\limsup_{n\to\infty}\left(||B_n||+||\mathbf{1}-A_nA_n^*||\right)^{1/2n}\le R^{-1}$ cannot hold.

Assume now that all the weights with $1>|z_j|>R^{-1}$ are canonical. Then all $u^{(n)}$'s are entire by Lemma \ref{inductive}.

Now let us fix $R_1$ and $R_2$ with $1<R_2<R_1<R$. By Lemma \ref{stripping2} there exists $n_0$ such that zeros of $u^{(n)}$ in $\bbD$ all lie in $\{z\in\bbC : R_2^{-1}<|z|<1\}$ for every $n\ge n_0$. This means that $(u^{(n)}{}^\sharp)^{-1}$ and $N_n^\sharp$ are analytic in $(\bbC\cup\{\infty\})\setminus \bbD_{R_2}$, where $N_n$ is as before $ M^{(n)}(z)/z$. Now the arguments after \eqref{p3_eq4.28} work without changes and prove that \eqref{p3_eq4.29} holds. This estimate was the only ingredient that was used in the proof of Theorem \ref{p3_th8}. This proves Theorem~\ref{p3_th11} for the general case.
%
\end{proof}

\subsection{Results in terms of the perturbation determinant}\label{p3_sub34}
Assuming the Jost function exists, define the \textbf{perturbation determinant} by
\begin{equation*}
L(z)= u(z) u(0)^{-1}.
\end{equation*}
Clearly, $L(0)=\mathbf{1}$.
Note that by \eqref{p3_eq22} and $u^{(n)}(0)\to \mathbf{1}$ (see \eqref{p3_eq4.29}) we have
\begin{equation*}
u(0)=\stackrel{\curvearrowleft}{\prod_{n=1}^n}A_n^{-1}.
\end{equation*}

We can reformulate Theorems~\ref{p3_th7} and~\ref{p3_th8} as follows.

\begin{theorem}\label{p3_th12}
Let $L(z)$ be a polynomial obeying~\eqref{eq:symmetry} and
\begin{itemize}
\item[(i)] $L(z)$ is invertible on $\overline{\D}\setminus \{\pm1\}$;
\item[(ii)] if $\pm1$ are zeros, they are simple;
\item[(iii)] $L(0)=\mathbf{1}$.
\end{itemize}
Then $L$ is the perturbation determinant for some Jacobi matrix (asymptotic to type $1$), and each such matrix obeys $\mathbf{1}-A_nA_n^*=B_n=\mathbf{0}$ for all large $n$.
\end{theorem}

\begin{theorem}\label{p3_th13}
Let $L(z)$ be analytic in $\{z\mid |z|<R\}$ for some $R>1$ and obeys~\eqref{eq:symmetry} and (i)--(iii) from Theorem~{\ref{p3_th12}}, then $L$ is the perturbation determinant for some Jacobi matrix (asymptotic to type $1$), and each such matrix has
\begin{equation*}
\limsup_{n\to\infty}\left(||B_n||+||\mathbf{1}-A_nA_n^*||\right)^{1/2n}\le R^{-1}.
\end{equation*}
\end{theorem}

\begin{remarks} 1. It is clear from the proof that the corresponding measure in the above two theorems (as well as in  the two theorems below) is not uniquely defined, but all possible $d\gamma$'s are related by $d\gamma_1=v^* d\gamma_2 v$ for constant unitaries $v$.

2. In other words, every two Jacobi matrices having the same perturbation determinant are related by $\widetilde{\calJ}=U\calJ U^{-1}$, where $U$ is an $l\times l$ block diagonal unitary $U=\sigma_1\oplus\sigma_2\oplus\sigma_3\oplus\ldots$, where $\sigma_n$ are unitary with $\lim_{n\to\infty}\sigma_n=\mathbf{1}$, and $\sigma_1$ is allowed to be different from $\mathbf{1}$.
\end{remarks}
\begin{proof}[Proofs]
Pick any unitary $\sigma$ and let $u(z)=L(z)\sqrt{H} \sigma$, where
$$H=\frac2\pi \int_{0}^\pi \sin^2\theta\,\left[L(e^{i\theta})^* L(e^{i\theta})\right]^{-1}d\theta\ge \bdnot.$$
 Then
$$
\frac2\pi \int_{0}^\pi \sin^2\theta\,\left[u(e^{i\theta})^* u(e^{i\theta})\right]^{-1}d\theta=\bdone,
$$
and so Theorems \ref{p3_th7}, \ref{p3_th8} apply.
\end{proof}
\medskip

Now assume there are bound states.

Lemma~\ref{p3_lm00} gives necessary and sufficient conditions for the unique existence of a nonnegative definite solution of 
\begin{align}
&\frac{z_j}{z_j^{-1}-z_j}{w}_j f(1/{\bar{z}_j})^*=-(z_j-z_j^{-1})\res_{z= z_j} f(z)^{-1}\label{p3_eq64},\\
&\ran {w}_j=\ran \res_{z= z_j} f(z)^{-1}\label{p3_eq65}
\end{align}
(compare it with \eqref{p3_eq3.43} and \eqref{p3_eq23}).
These conditions are~\eqref{eq:cond1}, ~\eqref{eq:cond2},~\eqref{eq:cond3} with $A=f(1/{\bar{z}_j})^*$ and $B=-(z_j-z_j^{-1})\res_{z= z_j} f(z)^{-1}$.


\begin{theorem}\label{p3_th14}
A polynomial $L(z)$ is the perturbation determinant for some Jacobi matrix with $\mathbf{1}-A_nA_n^*=B_n=\mathbf{0}$ for all large $n$ if and only if it obeys~\eqref{eq:symmetry} and
\begin{itemize}
\item[(i)] $L(z)$ is invertible on $(\overline{\D}\setminus \R)\cup\{0\}$;
\item[(ii)] all zeros on $\overline{\D}\cap\R$ are simple;
\item[(iii)] for each zero $z_j$ in $\D$, ~(\ref{p3_eq64})--(\ref{p3_eq65}) has a unique nonnegative definite solution ${w}_j$ (see Lemma~\ref{p3_lm00});
\item[(iv)] $L(0)=\mathbf{1}$.
\end{itemize}
\end{theorem}

\begin{theorem}\label{p3_th15}
Let $L(z)$ be analytic in $\{z\mid |z|<R\}$ for some $R>1$. $L(z)$ is the perturbation determinant for some Jacobi matrix with $\limsup_{n\to\infty}\left(||B_n||+||\mathbf{1}-A_nA_n^*||\right)^{1/2n}\le R^{-1}$ if and only if it obeys ~\eqref{eq:symmetry} and (i), (ii), (iv), and (iii) for every $z_j$ with $1>|z_j|>R^{-1}$.

\end{theorem}
\begin{proof}[Proofs]
Denote $v_j$ to be the nonnegative solutions of (\ref{p3_eq64})--(\ref{p3_eq65}) corresponding to $1>|z_j|>R^{-1}$. For the rest of $z_j$'s pick any nonnegative $v_j$. Let ${w}_j =\sigma ^* H^{-1/2}  v_j H^{-1/2} \sigma \ge\mathbf{0}$, and $u(z)=L(z)\sqrt H \sigma$, where $\sigma$ is any unitary matrix, and
$$
H=\sum_j v_j+\frac2\pi \int_{0}^\pi \sin^2\theta\,\left[L(e^{i\theta})^* L(e^{i\theta})\right]^{-1}d\theta \ge \bdnot.
$$
Then
$$
\sum_j w_j+\frac2\pi \int_{0}^\pi \sin^2\theta\,\left[u(e^{i\theta})^* u(e^{i\theta})\right]^{-1}d\theta= \sigma ^* H^{-1/2} H H^{-1/2} \sigma=\bdone.
$$
Moreover, ${w}_j$ solves (\ref{p3_eq64})--(\ref{p3_eq65}) with $f$ replaced by $u$ for every $1>|z_j|>R^{-1}$. This means that the condition (iii) of Theorem \ref{p3_th10}/\ref{p3_th11} holds, and all the weights for $z_j$ with $1>|z_j|>R^{-1}$ are canonical. Thus Theorems \ref{p3_th10}/\ref{p3_th11} apply and we are done.
\end{proof}


\section{Meromorphic Continuations of Matrix Herglotz Functions and Perturbations of the Free Case}
\begin{subsection}{Proof of Theorems \ref{p3_th17} and \ref{p3_th18}}
\begin{proof}[Proof of Theorem \ref{p3_th17}]
(I)$\Rightarrow$(II) Assume (I) holds. (A) follows from Theorem \ref{p3_th5} (viii).
(B) follows from Theorem \ref{p3_th5} (vi) and (v).
(C) is immediate from \eqref{p3_eq25}.

Now let us show
(D). First of all, it is a straightforward calculation to see that for any $F$ with a first order pole,
\begin{equation}\label{p3_eq_res}
\res_{z=\bar{z}^{-1}_0} F^\sharp (z)= -\frac{1}{\bar{z}_0^2} (\res_{z=z_0} F(z))^*.
\end{equation}

Since $u(z;\calJ)$ is analytic at $z_j^{-1}$, then using \eqref{p3_eq25},
\begin{equation*}
\bdnot=\res_{z=z_j^{-1}} u(z;\calJ) = (z_j^{-1}-z_j) \res_{z=z_j^{-1}} u^\sharp(z;\calJ)^{-1} (M(z_j^{-1})-M^\sharp(z_j^{-1}))^{-1},
\end{equation*}
which implies
\begin{equation*}\label{p3_eq3.61n}
\ran (M(z_j^{-1})-M^\sharp(z_j^{-1}))^{-1} \subseteq \ker \res_{z=z_j^{-1}} u^\sharp(z;\calJ)^{-1}.
\end{equation*}
Now,
\begin{align*}
\ker \res_{z=z_j^{-1}} u^\sharp(z;\calJ)^{-1}=\ker \res_{z=z_j} u(z;\calJ)^{-1}{}^*=\ran u(z_j;\calJ)^*&=\ker u(z_j;\calJ){}^\perp  \\
&=\ran \widetilde{w}_j{}^\perp = \ran\res_{z=z_j}M(z)^\perp,
\end{align*}
and $\ran (M(z_j^{-1})-M^\sharp(z_j^{-1}))^{-1}=\big[\ker(M(z_j^{-1})-M^\sharp(z_j^{-1}))^{-1}\big]^{\perp}$ since $M$ is Hermitian on the real line. This gives \eqref{p3_eq5.4}.

Note that $(M(z)-M^\sharp(z))^{-1} M(z)=\bdone+(M(z)-M^\sharp(z))^{-1} M^\sharp(z)$ is analytic at $z_j^{-1}$ since $(M(z)-M^\sharp(z))^{-1} M^\sharp(z)$ is analytic at $z_j^{-1}$ by \eqref{p3_eq5.4}.

Now, by \eqref{p3_eq22}, $u(z;\calJ)M(z)$ must be analytic at $z_j^{-1}$.
Then
using \eqref{p3_eq25},
\begin{equation*}
\bdnot=\res_{z=z_j^{-1}} u(z;\calJ) M(z)=  (z_j^{-1}-z_j) \res_{z=z_j^{-1}}   u^\sharp(z;\calJ)^{-1} (M(z)-M^\sharp(z))^{-1} M(z),
\end{equation*}
which implies $\ran (M(z_j^{-1})-M^\sharp(z_j^{-1}))^{-1} M(z_j^{-1}) \subseteq \ker \res_{z=z_j^{-1}}   u^\sharp(z;\calJ)^{-1}=\ran u^\sharp (z_j^{-1};\calJ) =\ran \widetilde{w}_j {}^\perp$, which is \eqref{p3_eq5.5}.

\bigskip

(II)$\Rightarrow$(I) Now assume (A)--(D) holds. Because of (A), $M$ has only finitely many poles $\{z_j\}$ in $\bbD$, all of which are real and simple since $M$ is Herglotz (see \cite{Ges}). Let $\widetilde{w}_j=-\res_{z=z_j} M(z)$.

Now we construct a function $u$ as described in Theorem \ref{p3_th7} and the remarks after it. First, there exists an outer function $O$ satisfying \eqref{p3_eq3.56}  by the Wiener--Masani
theorem (Lemma \ref{p2_lm2}) since Szeg\H{o}'s condition \eqref{p3_eq3.57} trivially holds. Then form a matrix-valued Blashcke product $B=\prod_j B_{z_j,s_j,U_j}$ with $s_j=\dim \ran \widetilde{w}_j$, where we pick unitary matrices $U_j$ so that $\ker B(z_j)O(z_j)=\ran \widetilde{w}_j$ (see Lemma \ref{product}). Now put $u(z)=B(z)O(z)$, which is an $\bbH^2(\bbD)$-function.
%

Define
\begin{equation}\label{p3_eq5.6}
\widehat{u}(z)=(z-z^{-1}) u^\sharp(z){}^{-1} (M(z)-M^\sharp(z))^{-1}, \quad 1<|z|<R.
\end{equation}
Since by the construction $u(e^{i\theta})^* u(e^{i\theta})=\sin\theta (\imag M(e^{i\theta}))^{-1}$, we have $\widehat{u}(e^{i\theta})=u(e^{i\theta})$, where the values of $u,\widehat{u}$ on $\partial \bbD$ are meant in the sense of nontangential limits. Now note by (C), $\sin\theta(\imag M(e^{i\theta}))^{-1}$ is continuous, and therefore $\sup_{z\in\partial\bbD}||u(z)||<\infty$. By the Smirnov maximum principle for matrix-valued functions (see \cite{Ginz3}), $\sup_{z\in\bbD}||u(z)||\le \sup_{z\in\partial\bbD}||u(z)||<\infty$, i.e., $u$ is bounded on $\bbD$. Note that 
$u^{-1}$ is bounded on a neighborhood of any point of $\partial\bbD\setminus \{\pm1\}$, and then so is $\widehat{u}$ by \eqref{p3_eq5.6}. Therefore 
Schwarz reflection principle allows us to conclude that $\widehat{u}$ is a meromorphic continuation of $u$. Since $u$ is bounded on $\overline{\bbD}$, $\pm1$ must be removable singularities.

Note that by (B), $M(z)-M^\sharp(z)$ in regular on $\partial\bbD\setminus\{\pm1\}$ with at most simple poles at $\pm1$. Therefore \eqref{p3_eq5.6} proves that $u$ has no zeros on $\partial\bbD\setminus\{\pm1\}$ with at most simple zeros at $\pm1$.

Thus $u$ satisfies all of the conditions of Theorem \ref{p3_th_constr} (with $w_j=(z_j^{-1}-z_j)z_j^{-1} \widetilde{w}_j$), and it's clear that the unique measure $\mu$ of Theorem \ref{p3_th_constr} is the measure corresponding to $M$. In order to apply Theorem \ref{p3_th11} we need to show that $u$ is analytic (rather than just meromorphic) in $\bbD_R$, and that the weights for those $z_j$ with $1>|z_j|>R^{-1}$ are canonical.

\eqref{p3_eq5.6} shows that singularities of $u$ can only happen at $z_j^{-1}$, in which case they are simple poles.  Note that \eqref{p3_eq5.4} can be rewritten as
\begin{equation*} \label{p3_eq5.7}
\begin{aligned}
\ran (M(z_j^{-1})-M^\sharp(z_j^{-1}))^{-1}& = \left( \ker (M(z_j^{-1})-M^\sharp(z_j^{-1}))^{-1} \right)^\perp \supseteq \ran \widetilde{w}_j {}^\perp = \ker u(z_j) {}^\perp \\
&=\ran u(z_j) = \ran u^\sharp(z^{-1}_j) =\ker \res_{z=z_j^{-1}}u^\sharp(z)^{-1},
\end{aligned}
\end{equation*}
where in the second-to-last equality we used \eqref{p3_eq_res}. This and \eqref{p3_eq5.6} imply
$$
\res_{z=z_j^{-1}} u(z)=  (z_j-z_j^{-1}) \res_{z=z_j^{-1}} u^\sharp(z)^{-1} (M(z_j^{-1})-M^\sharp(z_j^{-1}))^{-1}=\bdnot,
$$
i.e., there is no pole at $z_j$, i.e., $u$ is analytic in $\bbD_R$.

By the remark after \eqref{one} we will establish that all the weights are canonical if we show that $u^{(1)}(z)=z^{-1}u(z)M(z)A_{n+1}$ is analytic at $z_j^{-1}$.
This is what \eqref{p3_eq5.5} is for.

First of all, note that $\ran\res_{z=z_j^{-1}} M^\sharp(z)=\ran\res_{z=z_j}M(z)$ (just use \eqref{p3_eq_res} and $\widetilde{w}_j=\widetilde{w}_j^*$), so \eqref{p3_eq5.4} implies that $(M(z)-M^\sharp(z))^{-1} M(z)=\bdone+(M(z)-M^\sharp(z))^{-1} M^\sharp(z)$ is analytic at $z_j^{-1}$. This justifies that the use of the expression in \eqref{p3_eq5.5}.
Now note that \eqref{p3_eq5.5} can be rewritten as
\begin{equation*}\label{p3_eq5.75}
\ran (M(z^{-1}_j)-M^\sharp(z^{-1}_j))^{-1} M(z^{-1}_j) \subseteq \ran \widetilde{w}_j {}^\perp = \ker \res_{z=z_j^{-1}}u^\sharp(z)^{-1},
\end{equation*}
which implies that $u^\sharp(z_j^{-1})^{-1}(M(z^{-1}_j)-M^\sharp(z^{-1}_j))^{-1} M(z^{-1}_j)$ is analytic. By \eqref{p3_eq5.6} this is $u(z_j^{-1})M(z_j^{-1})$.

Theorem \ref{p3_th11} applies, giving (I).
\end{proof}

\begin{proof}[Proof of Theorem \ref{p3_th18}]
That (I) implies (II) is clear from \eqref{p3_eq22} and the fact that $u$ and $u^{(1)}$ are polynomials.

Assume (II) holds. Then, going through the proof of the previous theorem, note that $u$ is entire and by \eqref{p3_eq5.6} grows at most polynomially. Therefore it is a polynomial, and so Theorem \ref{p3_th10} applies.
\end{proof}
\end{subsection}

In the remarks after Theorems \ref{p3_th17} and \ref{p3_th18} we mentioned that condition (D) can be restated in a better-looking form in some special cases. Let us prove it here.
\begin{proposition}\label{propos}
\begin{itemize}
\item If $M$ has a pole of the first order at $z_j^{-1}$ then (D) is equivalent to
\begin{gather}
\label{p3_eq3.55n}
\ran \widetilde{w}_j \subseteq \ran (\widetilde{w}_j-z_j^2 \widetilde{q}_j),\\
\label{p3_eq3.56n}
\ran \widetilde{w}_j \cap \ran \widetilde{q}_j=\varnothing,
\end{gather}
where $\widetilde{w}_j=-\res_{z=z_j} M(z)$, $\widetilde{q}_j=\res_{z= z^{-1}_j} M(z)$.
\item If $l=1$, then (D) is equivalent to the condition that $M$  has no simultaneous singularities at points $z_j$ and $z_j^{-1}$.
\end{itemize}
\end{proposition}
\begin{proof}
If $M$ has a first order pole at $z_j^{-1}$, then we can apply Lemma \ref{p4_lm0}  to the analytic function $(M-M^\sharp)^{-1}$ and see that \eqref{p3_eq5.4} can be rewritten as
\begin{equation*}\label{p3_eq3.62n}
\ran \widetilde{w}_j \subseteq \ran\res_{z=z_j^{-1}} M(z)-M^\sharp(z)=\ran (\widetilde{q}_j-\frac{1}{z_j^2}\widetilde{w}_j),
\end{equation*}
where we used \eqref{p3_eq_res} and the fact that  and $\widetilde{w}_j$ and $\widetilde{q}_j$ are Hermitian.

Now note that \eqref{p3_eq5.5} is equivalent to
\begin{equation*}
\bdnot=\res_{z=z_j^{-1}} u(z;\calJ) M(z)=  u(z_j^{-1};\calJ) \res_{z=z_j^{-1}} M(z),
\end{equation*}
which means
\begin{equation}\label{p3_eq3.64n}
\begin{aligned}
\ran \widetilde{q}_j \subseteq \ker u(z_j^{-1};\calJ)&= \ran \res_{z=z_j^{-1}} u(z;\calJ)^{-1}= \ran \res_{z=z_j^{-1}} (M(z)-M^\sharp(z)) u^\sharp(z;\calJ) \\
&= \ran(\widetilde{q}_j-\frac{1}{z_j^2}\widetilde{w}_j) u(z_j;\calJ)^* = \ran \widetilde{q}_j u(z_j;\calJ)^*
\end{aligned}
\end{equation}
where we successively used here: Lemma \ref{p4_lm0}, \eqref{p3_eq25}, \eqref{p3_eq_res}, and \eqref{p3_eq23}. Finally, note that \eqref{p3_eq3.64n} is equivalent to $\ker \widetilde{q}_j\supseteq \ker u(z_j;\calJ) \widetilde{q}_j$, i.e., $\ran \widetilde{q}_j \cap \ker u(z_j;\calJ)=\varnothing$, which is \eqref{p3_eq3.56n} by \eqref{p3_eq23}.

\bigskip

Now let $l=1$, and assume $M$ is pole of order $1$ at $z_j\in\bbD$ (it cannot have higher order poles there), and of order $k\ge1$ at $z_j^{-1}$. Then $\lim_{z\to z_j^{-1}}(1-\frac{M^\sharp(z)}{M(z)})$ is finite, so $\lim_{z\to z_j^{-1}}(1-\frac{M^\sharp(z)}{M(z)})^{-1}$  is nonzero (and it actually cannot be infinite by \eqref{p3_eq5.4}). Therefore the right-hand side of \eqref{p3_eq5.5} becomes $\left(\ran \lim_{z\to z_j^{-1}}(1-\frac{M^\sharp(z)}{M(z)})^{-1} \right)^\perp=\{0\}$. But the left-hand side is $\bbC$, a contradiction.
\end{proof}


\end{document}